\documentclass[final,1p,times]{elsarticle}

\usepackage{amssymb}
\usepackage{amsthm}
\usepackage{amsmath,amssymb,amsopn,amsfonts,mathrsfs,amsbsy,amscd}
\usepackage{longtable}
\usepackage{caption}
\usepackage{multirow}
\usepackage{lineno}
\newcommand{\g}{\mathfrak{g}}

\newcommand{\om}{\omega}
\newcommand{\esp}{\quad\mbox{and}\quad}

\newcommand{\e}{\dot{e}}

\newcommand{\G}{{\mathfrak{g}}}

\newcommand{\B}{{\cal B}}

\newcommand{\al}{\alpha}
\newcommand{\be}{\beta}

\newcommand{\R}{\mathbb{R}}

\newtheorem{theo}{Theorem}[section]

\newtheorem{co}{Corollary}[section]

\newtheorem{remark}{Remark}

{%
	\stepcounter{equation}
	\begin{equation*}}{%
		\leqno(\arabic{equation})
	\end{equation*}
}

\begin{document}
	
	\begin{frontmatter}

		\title{The classification of left-invariant para-K\"ahler structures on simply connected four-dimensional Lie groups
		}

		\author[]{ M. W. Mansouri and A. Oufkou}
		\address[]{Universit\'e Ibno Tofail\\ Facult\'e des Sciences. K\'enitra-Maroc\\e-mail: mansourimohammed.wadia@uit.ac.ma\\ ahmed.oufkou@uit.ac.ma}
		
		
		
		
		\begin{abstract}  We give a complete classification  of  left
			invariant   para-K\"ahler structures on four-dimensional  simply connected
			Lie groups up to an  automorphism. As an application we discuss  some curvatures  properties of the canonical connection associated to these structures as   flat, Ricci flat and existence of Ricci solitons.

		\end{abstract}
		

	\end{frontmatter}
	
	\section{Introduction and main results}
	
	An almost para-complex structure on a $2n$-dimensional manifold $M$ is a field $K$ of endomorphisms of the tangent
	bundle $TM$ such that $K^2 = Id_{TM}$ and the two eigendistributions $T^{\pm}M := ker(Id\pm K)$ have the same rank.  An almost
	para-complex structure $K$ is said to be integrable if the distributions $T^{\pm}M$ are involutive. This is equivalent to the
	vanishing of the Nijenhuis tensor $N_K$ defined by
	\[N_K(X, Y ) = [X,Y] + [KX,KY] -K[KX,Y]-K[X,KY ],\]
	for vector fields $X$, $Y$ on $M$. In such a case $K$ is called a para-complex structure.
	A para-K\"ahler structure on a manifold $M$ is a pair $(\langle .,.\rangle,K)$ where $\langle .,.\rangle$ is a
	pseudo-Riemannian metric and $K$ is a parallel skew-symmetric para-complex
	structure. If $(\langle .,.\rangle,K)$ is a para-K\"ahler structure on $M$, then $\om=\langle .,.\rangle\circ K$ is a symplectic
	structure and the $\pm 1-$eigendistributions $T^\pm M$ of $K$ are two integrable
	$\om$-Lagrangian distributions. Due to this, a para-K\"ahler structure can be
	identified with a bi-Lagrangian structure $(\om, T^\pm M)$ where $\om$ is a symplectic
	structure and $T^\pm M$ are two integrable Lagrangian distributions. Moreover the Levi-Civita connection associate to
	neutral metric $\langle .,.\rangle$ coincides with the canonical connection associate to bi-Lagrangian structure (the unique
	symplectic connection  with parallelizes both foliations \cite{He}). For a survey on paracomplex geometry see  [6] and for
	background on bi-Lagrangian structures and their associated
	connections, the survey \cite{E-S-T} is a good starting point and contains
	further references (See as well \cite{A-M-T} and \cite{Bo}).
	
	Suppose now that $M$ is a Lie group $G$ and  $\omega$, $\langle .,.\rangle$ and $K$ are left invariant. If we denote
	by $\g$ the Lie algebras of $G$, then $(\langle .,.\rangle,K)$ is  determined by is restrictions to the Lie algebra $\g$. In this situation, $(\g, \langle .,.\rangle_e, K_e)$  or $(\g, \omega_e, K_e)$ is called a para-K\"ahler Lie algebra (e is unit of
	$G$), in the rest of this paper  a para-K\"ahler Lie algebra will be noted  $(\g,\omega,K)$. Recall that two para-K\"ahler Lie algebras $(\g_1,\omega_1,K_1)$ and $(\g_2,\omega_2,K_2)$ are said to be
	equivalent if there exists an isomorphism of Lie algebras
	$T :\g_1\longrightarrow \g_1$ such as $T^*\omega_2=\omega_1$ and $T_*K_1=K_2$. 
	Para-K\"ahler (bi-Lagrangian) structures on Lie algebras in general have been studied, for example, in \cite{Ba}, \cite{B-B} and \cite{B-M}. In \cite{Ha}, there is a study the
	existences of bi-Lagrangian structures on symplectic nilpotent Lie algebras of dimension $2$,$4$ and $6$.
	A first classification of  para-K\"ahler structures on four-dimensional Lie algebras was obtained by  Calvaruso in 
	\cite{Ca}. Another  description based on the classification of symplectic Lie algebras is  proposed by   Smolentsev and  Shagabudinova in \cite{S-S}.
	Benayadi and Boucetta  provide in \cite{B-B} a new characterization of para-K\"ahler Lie algebras using left symmetric bialgebras inroduced by Bai in \cite{Ba}. Based on this characterization  we propose in this paper, the  classification of para-K\"ahler structures on four-dimensional Lie algebras. Notice that our classification is more complete and precise than the other classifications existing in the literature.
	
	\textit{Notations}:  For $\{e_1, e_2,e_3,e_4\}$  a basis of $\G$, we denote by $\{e^1, e^2, e^3,e^4\}$ the
	dual basis on $\G^\ast$ and  $e^{ij}$  the two-form $e^i\wedge e^j$, $\e^{ij}$ is  the symmetric two-form $e^i \odot e^j$ and $E_{ij}$ is the endomorphism which sends $e_j$ to $e_i$ and vanishes on $e_k$ for $k\not=j$.
	
	The para-K\"ahler  Lie algebras $(\g, \langle .,.\rangle, K)$ is necessarily symplectic Lie algebra $(\g,\om)$. It is well known that a symplectic four-dimensional Lie
	algebra is necessarily solvable. The classification of symplectic
	four-dimensional  Lie algebras $(\G,\omega)$ is given by the following
	Table  (see \cite{O}).
	{\renewcommand*{\arraystretch}{0.9	}
		\begin{center}
			\begin{longtable}{llc}
				\hline
				Case & No vanishing brackets & $\omega$ \\
				\hline
				$\mathfrak{r} \mathfrak{h}_3$ &$[e_1,e_2]=e_3$&$e^{14}+e^{23}$ \\
				\hline
				$\mathfrak{r}\mathfrak{r}_{3,0}$ & $[e_1,e_2]= e_2$&$e^{12}+e^{34}$ \\
				\hline
				$\mathfrak{r}\mathfrak{r}_{3,-1}$ & $[e_1,e_2]= e_2$, $[e_1, e_3] =-e_3$&$e^{14}+e^{23}$ \\
				\hline
				$\mathfrak{r}\mathfrak{r}^\prime_{3,0}$ &$[e_1,e_2]=-e_3$, $[e_1,e_3]=e_2$&$e^{14}+e^{23}$ \\
				\hline
				$\mathfrak{r}_2\tau_2$ &$[e_1,e_2] =e_2$, $[e_3,e_4]=e_4$&$e^{12}+\mu e^{13}+e^{34}$ \\
				\hline
				\multirow{2}{*}{$\mathfrak{r}^\prime_2$} &$[e_1,e_3]= e_3$, $[e_1,e_4]=e_4$, &\multirow{2}{*}{$e^{14}+e^{23}$} \\
				&$[e_2,e_3] = e_4$, $[e_2,e_4] = -e_3$&\\
				\hline
				$\mathfrak{n}_4$& $[e_4,e_1]=e_2$, $[e_4,e_2]=e_3$ &$e^{12}+e^{34}$ \\
				\hline
				$\mathfrak{r}_{4,0}$ &$[e_4,e_1]=e_1$, $[e_4,e_3]=e_2$ &$e^{14}\mp e^{23}$ \\
				\hline
				$\mathfrak{r}_{4,-1}$ &$[e_4,e_1]=e_1$, $[e_4,e_2]=-e_2$, $[e_4,e_3]=e_2-e_3$&$e^{13}+e^{24}$ \\
				\hline
				$\mathfrak{r}_{4,-1,\beta}$ &$[e_4,e_1]=e_1$, $[e_4,e_2]=-e_2$, $[e_4,e_3]=\beta e_3$&$e^{12}+e^{34}$ \\
				\hline
				$\mathfrak{r}_{4,\alpha,-\alpha}$&$[e_4,e_1]=e_1$, $[e_4,e_2]=\alpha e_2$, $[e_4,e_3]=-\alpha e_3$& $e^{14}+e^{23}$\\
				\hline
				$\mathfrak{r}^\prime_{4,0,\delta}$ &$ [e_4, e_1]=e_1$, $[e_4, e_2]=-\delta e_3$, $[e_4, e_3]=\delta e_2$&$e^{14}\mp e^{23}$ \\
				\hline
				\multirow{2}{*}{$\mathfrak{d}_{4,1}$}&$[e_1,e_2]=e_3$, $[e_4,e_3]= e_3$, &$e^{12}-e^{34}$\\
				&$[e_4, e_1]=e_1$&$e^{12}-e^{34}+e^{24}$\\
				\hline
				\multirow{2}{*}{$\mathfrak{d}_{4,2}$}&$[e_1,e_2]=e_3$, $[e_4,e_3]= e_3$,& $e^{12}-e^{34}$\\
				& $[e_4, e_1]=2e_1$, $[e_4, e_2]=-e_2$&$e^{14}\mp e^{23}$\\
				
				\hline
				\multirow{2}{*}{$\mathfrak{d}_{4,\lambda}$}&$[e_1,e_2]=e_3$, $[e_4,e_3]= e_3$, &\multirow{2}{*}{$e^{12}-e^{34}$} \\
				&$[e_4, e_1] =\lambda e_1$, $[e_4, e_2]= (1-\lambda)e_2$&\\
				\hline
				\multirow{2}{*}{$\mathfrak{d}^\prime_{4,\delta}$}&$[e_1,e_2] = e_3$, $[e_4, e_1] = \frac{\delta}{2}e_1-e_2$,&\multirow{2}{*}{$\mp(e^{12}-\delta e^{34})$}\\
				& $[e_4,e_3]=\delta e_3$, $[e_4,e_2] = e_1+\frac{\delta}{2}e_2$&\\
				\hline
				\multirow{2}{*}{$\mathfrak{h}_4$}&$[e_1,e_2]=e_3$, $[e_4, e_3]=e_3$, &\multirow{2}{*}{$\mp(e^{12}-e^{34})$}\\
				&$[e_4, e_1]=\frac{1}{2}e_1$, $[e_4,e_2] = e_1 +\frac{1}{2}e_2$&\\
				\hline
				\caption{ Symplectic four-dimensional Lie algebras\\  $(\mu\geq0$, $-1\leq\beta<1$, $-1<\alpha<0$, $\delta>0$ and $\lambda\geq\frac{1}{2}$,  $\lambda\not=1,2 ) $.}
				\label{tab1}
			\end{longtable}
			
		\end{center}

		Our main result is the following.
		\begin{theo}\label{1-1}Let $(\g,\om,K)$ be a four-dimensional para-K\"ahler Lie algebra. Then $(\g,\om,K)$ is isomorphic to one of the following Lie algebras with the given para-K\"ahler structures:
			
			\textbf{Lie algebra $\mathfrak{r} \mathfrak{h}_3$}
			\begin{enumerate}
				\item[]For $\om=e^{14}+e^{23}$ 
				\begin{itemize}
					\item[]$K_1=-E_{11}+E_{21}+E_{22}-E_{33}-E_{43}+E_{44}$	
					\item[]$K_2=\mp(E_{11}-E_{22}+E_{33}-E_{44})$
									\end{itemize}
			\end{enumerate}		
			\textbf{Lie algebra $\mathfrak{r} \mathfrak{r}_{3,0}$} 
			\begin{enumerate}
				\item[]For $\om=e^{12}+e^{34}$ 
				\begin{itemize}
					\item[]$K_1=-E_{11}+E_{22}-E_{33}+E_{44}$
					\item[]$K_2=E_{11}+xE_{12}-E_{22}+E_{33}-E_{44}$
				\end{itemize}
			\end{enumerate}  
			\textbf{Lie algebra $\mathfrak{r}\mathfrak{r}_{3,-1}$}
			\begin{enumerate}
				\item[]For $\om=e^{14}+e^{23}$
				\begin{itemize}
					\item[]$K_1=\mp(E_{11}+E_{22}\pm E_{23}-E_{33}-E_{44})$
					\item[]$K_2=E_{11}+E_{22}+xE_{14}-E_{33}-E_{44}$
					\item[]$K_{3}=E_{11}-E_{22}\pm E_{32}+E_{33}-E_{44}$
					\item[]$K_{4}=E_{11}-E_{22}+E_{33}-E_{44}$
				\end{itemize}	
			\end{enumerate}
			\textbf{Lie algebra $\mathfrak{r}_{2}\mathfrak{r}_{2}$} 
			\begin{enumerate}
				\item[]For $\omega=e^{12}+\mu e^{13}+e^{34},\; (\mu>0)$
				\begin{itemize}
			\item[]$K_1=-E_{11}+E_{22}+E_{33}-E_{44}$
					\item[]$K_2=\mp(E_{11}-2E_{13}-E_{22}-E_{33}+2E_{42}+E_{44})$
				\end{itemize}
				\item[]For $\omega=e^{12}+e^{34}$
				\begin{itemize}
					\item[]$K_1=E_{11}-E_{22}-2E_{24}+2E_{31}-E_{33}+E_{44}$
					\item[]$K_2=-E_{11}+xE_{12}+E_{13}+xE_{14}+E_{22}+xE_{32}+xE_{34}-E_{42}+\frac{1}{x}E_{43}$
					\item[]$K_3=-E_{11}+E_{22}-E_{33}+E_{44}$
					\item[]$K_4=-E_{11}+E_{22}+E_{33}+xE_{43}-E_{44}$
					\item[]$K_5=E_{11}-2E_{13}-E_{22}-E_{33}+2E_{42}+E_{44}$
					\item[]$K_6=E_{11}+xE_{12}-E_{22}+E_{33}+yE_{34}-E_{44}$
					\item[]$K_7=E_{11}+xE_{12}+xE_{14}-E_{22}+xE_{32}+E_{33}+xE_{34}-E_{44}$
				\end{itemize}
			\end{enumerate}
			\textbf{Lie algebra $\mathfrak{r}_{2}'$} 
			\begin{enumerate}
				\item[]For $\omega=e^{14}+e^{23} $
				\begin{itemize}	
					\item[]$K_{1}=E_{11}+xE_{14}-E_{22}-\frac{4}{x}E_{32}+E_{33}-E_{44}$
					\item[]$K_{2}=-E_{11}-E_{22}+E_{33}+E_{44}$
					\item[]$K_{3}=xE_{11}+2yE_{12}+(1-x)E_{13}-2yE_{14}-2yE_{21}+xE_{22}+2yE_{23}+(1-x)E_{24}+(1+x)E_{31}+2yE_{32}-xE_{33}-2yE_{34}-2yE_{41}+(1+x)E_{42}+2yE_{43}-xE_{44}$
					\item[]$K_{4}=E_{11}+xE_{14}+E_{22}-E_{33}-E_{44}$
				\end{itemize}
			\end{enumerate}
			\textbf{Lie algebra $\mathfrak{r}_{4,0}$} 
			\begin{enumerate}
				\item[]For $\omega=e^{14}+ e^{23}$ or  $\omega=e^{14}-e^{23}$ 
				\begin{itemize}
				\item[]$K=\mp(E_{11}-E_{22}+E_{33}-E_{44})$
				\end{itemize}
			\end{enumerate}
			\textbf{Lie algebra $\mathfrak{r}_{4,-1}$}
			\begin{enumerate}
				\item[]For $\omega=e^{13}+e^{24}$
				\begin{itemize}
					\item[]$K_{1}=E_{11}+xE_{13}-E_{22}-E_{33}+E_{44}$
					\item[]$K_{2}=-E_{11}+E_{22}+E_{33}-E_{44}$
				\end{itemize}
			\end{enumerate}
			\textbf{Lie algebra $\mathfrak{r}_{4,-1,\beta} \;(-1<\beta<1)$} 
			\begin{enumerate}
				\item []For $\omega =e^{12}+e^{34}$
				\begin{itemize}
					\item []$K_{1}=E_{11}\mp E_{12}-E_{22}-E_{33}+E_{44}$
					\item []$K_{2}=E_{11}-E_{22}\mp E_{33}\pm E_{44}$
					\item []$K_{3}=E_{11}-E_{22}+xE_{34}+\frac{1}{x}E_{43}$
					\item []$K_{4}=E_{11}-E_{22}+E_{33}-E_{44}$
					\item []$K_{5}=-E_{11}\mp E_{12}+E_{22}+E_{33}-E_{44}$
									\end{itemize}
			\end{enumerate}			
			\textbf{Lie algebra $\mathfrak{r}_{4,-1,-1}$} 
			\begin{enumerate}
				\item []For $\omega=e^{12}+e^{34} $
				\begin{itemize}
					\item []$K_{1}=-E_{11}+xE_{21}+E_{22}+E_{23}-E_{33}+E_{41}+E_{44}$
					\item []$K_{2}=-E_{11}+E_{21}+E_{22}-E_{33}+E_{44}$
					\item []$K_{3}=-E_{11}-E_{21}+E_{22}-E_{33}+E_{44}$
					\item []$K_{4}=-E_{11}+E_{22}-E_{33}+xE_{43}+E_{44}$
					\item []$K_{5}=E_{11}\mp E_{12}-E_{22}-E_{33}+E_{44}$
					\item []$K_{6}=E_{11}-E_{22}-E_{33}+xE_{43}+E_{44}$
					\item []$K_{7}=-E_{11}+E_{22}+E_{33}-E_{44}$
					\item []$K_{8}=E_{11}\mp E_{21}-E_{22}+E_{33}-E_{44}$
								\end{itemize}
			\end{enumerate}
			\textbf{Lie algebra $\mathfrak{r}_{4,\al,-\al}$}\;($-1<\al<0$)
			\begin{enumerate}
				\item [] For $\omega=e^{14}+e^{23} $
				\begin{itemize}
					\item [] $K_{1}=-E_{11}+E_{22}+E_{23}-E_{33}+E_{44}$
					\item [] $K_{2}=-E_{11}+E_{22}-E_{23}-E_{33}+E_{44}$
					\item [] $K_{3}=\mp(-E_{11}+E_{22}-E_{33}+xE_{41}+E_{44})$
					\item [] $K_{4}=E_{11}+E_{22}-E_{33}-E_{44}$
					\item [] $K_{5}=\mp(E_{11}+E_{22}\mp E_{32}-E_{33}-E_{44})$
				\item [] $K_{6}=E_{11}-E_{22}-E_{23}+E_{33}-E_{44}$
					\item [] $K_{7}=E_{11}-E_{22}+E_{23}+E_{33}-E_{44}$				
					\item [] $K_{8}=-E_{11}-E_{22}-E_{32}+E_{33}+E_{44}$
					\item [] $K_{9}=-E_{11}-E_{22}+E_{33}+E_{44}$
				\end{itemize}
			\end{enumerate}
			
			\textbf{Lie algebra $\mathfrak{d}_{4,1}$} 
			\begin{enumerate}
				\item[]For $\om=e^{12}-e^{34}$
				\begin{itemize}
					\item[]$K_1=E_{11}\mp E_{12}-E_{22}-E_{33}+E_{44}$
					\item[]$K_2=\mp E_{11}\pm E_{22}-E_{33}+xE_{43}+E_{44}$
					\item[]$K_3=E_{11}-E_{22}+E_{33}-E_{44}$
					\item[] $K_4=-E_{11}\mp E_{12}+E_{22}+E_{33}-E_{44}$
					\item[]$K_5=E_{11}+E_{21}-E_{22}+xE_{23}+E_{33}-xE_{41}-E_{44}$
					\item[]$K_6=-E_{11}\mp E_{21}+E_{22}-E_{33}+E_{44}$
									\end{itemize}			
				\item[]For $\om=e^{12}-e^{34}+e^{24}$			
				\begin{itemize}
					\item[]$K_{1}=\mp(E_{11}+xE_{12}-E_{22}-E_{33}+E_{44})$
				\end{itemize}
			\end{enumerate}	
			\textbf{Lie algebra $\mathfrak{d}_{4,2}$}
			\begin{enumerate}
				\item[]For $\om=e^{12}-e^{34}$
				\begin{itemize}
					\item[]$K_1=E_{11}\mp E_{12}-E_{22}-E_{33}+E_{44}$
					\item[]$K_2=E_{11}-E_{22}-E_{33}+E_{44}$
					\item[]$K_3=E_{11}-E_{22}+E_{33}+xE_{43}-E_{44}$
				\end{itemize}
				\item[]For $\om=e^{14}-e^{23}$
				\begin{itemize}
					\item[]$K_1=-E_{11}-E_{22}+\frac{1}{x}E_{32}+E_{33}-2xE_{14}+E_{44}$
					\item[]$K_2=-E_{11}-E_{22}-2E_{31}+xE_{32}+E_{33}+2E_{24}+E_{44}$
					\item[]$K_3=E_{11}+E_{22}-2E_{31}+xE_{32}-E_{33}+2E_{24}-E_{44}$
					\item[]$K_4=E_{11}-E_{22}+xE_{12}+xE_{32}+E_{33}+xE_{14}+xE_{34}-E_{44}$
					\item[]$K_5=-E_{11}+E_{22}-E_{33}+xE_{41}+E_{44}$
					\item[]$K_6=E_{11}-E_{22}-2xE_{23}+E_{33}+xE_{41}-E_{44}$
					\item[] $K_7=-E_{11}+2xE_{21}+E_{22}-2xE_{23}-E_{33}-2xE_{41}-2xE_{43}+E_{44}$
				\end{itemize}
				\item[]For $\om=e^{14}+e^{23}$
				\begin{itemize}
					\item[]$K_1=-E_{11}-E_{22}+\frac{1}{x}E_{32}+E_{33}-2xE_{14}+E_{44}$
					\item[]$K_2=-E_{11}+E_{22}-E_{33}+xE_{14}+E_{44}$
					\item[]$K_3=E_{11}-E_{22}+xE_{23}+E_{33}-\frac{1}{2}xE_{14}-E_{44}$
				\end{itemize}
			\end{enumerate}	
			\textbf{Lie algebra $\mathfrak{d}_{4,\frac{1}{2}}$}
			\begin{enumerate}
				\item[]For $\omega =e^{12}-e^{34}$
				\begin{itemize}
					\item[]$K_1=E_{11}-E_{22}-E_{33}+xE_{43}+E_{44}$
					\item[]$K_2=E_{11}-E_{22}+E_{33}-E_{44}$
				\end{itemize}
			\end{enumerate}
			\textbf{Lie algebra $\mathfrak{d}_{4,\lambda}$}\; ($\lambda>\frac12$, $\lambda\not=1,2$)
			\begin{enumerate}
				\item[]For $\omega =e^{12}-e^{34}$
				\begin{itemize}
					\item[]$K_1=E_{11}\mp E_{12}-E_{22}-E_{33}+E_{44}$
					\item[]$K_2=-E_{11}+E_{21}+E_{22}-E_{33}+E_{44}$
					\item[]$K_3=-E_{11}-E_{21}+E_{22}-E_{33}+E_{44}$
					\item[]$K_4=E_{11}-E_{22}-E_{33}+xE_{43}+E_{44}$
					\item[]$K_5=-E_{11}+E_{22}-E_{33}+xE_{43}+E_{44}$
					\item[]$K_6=\mp E_{11}+xE_{12}\pm E_{22}+E_{33}-E_{44}$
									\end{itemize}
			\end{enumerate}
			\textbf{Lie algebra $\mathfrak{h}_{4}$}
			\begin{enumerate}
				\item[]For $\omega =\pm(e^{12}-e^{34})$
				\begin{itemize}
					\item[]$K_1=\mp(E_{11}-E_{22}-E_{33}+E_{44})$
								\end{itemize}
			\end{enumerate}
		\end{theo}

		\begin{co}
			The symplectic Lie  algebras $\mathfrak{r}\mathfrak{r}^\prime_{3,0}$ ,  $\mathfrak{n}_{4}$,  $\mathfrak{r}_{4,0,\delta}^{\prime}$ and  $\mathfrak{d}_{4,\delta}^{\prime}$    does not admit a  para-K\"ahler structure.
		\end{co}
		The paper is organized as follows. Section 2 contains the basic results which are essential to the classification of four-dimensional  para-K\"ahler   Lie algebras (proof of the Theorem \ref{1-1}).  Theorem \ref{2.1} and  Theorem \ref{main2} are the key steps in this proof.
		Section 3 is devoted to some  curvature properties of four-dimensional para-K\"ahler metrics. Section 4 contains  the  tables of Theorems \ref{main2} and the isomorphisms tables used in the proof of Theorem \ref{1-1}.

		The software Maple $18^\circledR$ has been used to check all needed calculations.
		\section{Proof of The  Theorem \ref{1-1}}
		
		In this section we begin  with a reminder of the new approach introduced by Benayadi and Boucetta  in \cite{B-B},    which characterizes the para-K\"ahler Lie algebras.

		Recall that, a para-K\"ahler  Lie algebra $(\g,\langle.,. \rangle,K)$ is carries a  Levi-Civita product, the product characterized by  Koszul's formula:
		\[2\langle u.v,w\rangle= \langle[u, v], w\rangle+ \langle [w,u], v\rangle+ \langle[w, v], u\rangle.\]
		The subalgebras $\g^{1}=\ker(K-Id_\g)$ and   $\g^{-1}=\ker(K+Id_\g)$  have the following properties, $\g^{1}$ and   $\g^{-1}$ are  isotropic  with respect to $\langle.,. \rangle$, Lagrangian with respect to $\om$ and checking that $\g=\g^{1}\oplus \g^{-1}$, moreover the restriction of the Levi-Civita  product  on $\g^{1}$ and $\g^{-1}$ induces a left symmetric structures. i.e.  for any $u$, $v$, $w\in\g^{1}$ (resp. $\g^{-1}$),
		\[ass(u,v,w) = ass(v,u,w)\]
		where $ass(u,v,w) = (u.v).w -u.(v.w)$. In particular, $\g^{1}$ and $\g^{-1}$ are left symmetric algebras.

		For any $u\in\g^{-1}$, let $u^*$ denote the element of $(\g^1)^*$   given by $u^*(v)=\langle u,v \rangle$. The map $u\longmapsto u^*$ realizes an isomorphism between $\g^{-1}$ and $(\g^1)^*$. Thus, we can identify  $(\g,\langle.,. \rangle,K)$  relative to the phase space $(\g^1\oplus(\g^1)^*,\langle.,.\rangle_0,K_0)$, where $\langle.,.\rangle_0$ and $K_0$ are  given by:
		\[\langle u+\al,v+\be\rangle_0 = \al(v)+\be(u)\esp K_0(u+\al) =u-\al.\]
		
		Both $\g^{1}$ and $(\g^{1})^*$ carry a left symmetric algebra structure. For any $u\in\g^{1}$ and for any $\al\in(\g^{1})^*$, we denote $L_u : \g^{1}\to\g^{1}$ and $L_\al: (\g^{1})^*\to (\g^{1})^*$ as the left multiplication by $u$ and $\al$, respectively, i.e., for any $v\in\g^{1}$ and any $\be\in(\g^{1})^*$,
		\[L_uv=u.v\esp   L_\al\be=\al.\be.\]
		
		The Levi-Civita product (and the Lie bracket) on $\g$  is determined entirely by their restrictions to $(\g^{1})^*$ and $\g^{1}$: For any $u\in\g^{1}$ and for any $\al\in\g^{1})^*$,
		\[u.\al=L^t_X\al \esp \al.u=- L^t_\al X.\]

		Conversely, let $U$ be a finite dimensional vector space and $U^*$ is its dual space. We suppose that both $U$ and $U^*$ have the structure of a left symmetric algebra. We extend the products on $U$ and $U^*$ to $U\oplus U^*$ for any $X$, $Y\in U$ and for any $\al$, $\be\in U^*$, by putting
		
		\begin{equation}\label{1}
			(X+\al).(Y+\be) = X.Y - L^t_\al Y- L^t_X\be + \al.\be.
		\end{equation}
		We say that two left symmetric products on $U$ and $U^*$  is Lie-extendible if the commutator of the product on $U\oplus U^*$ given by (1) is a Lie bracket. In this case we have the following theorem:

		\begin{theo} \label{2.1} \cite{B-B} Let $(U, .)$ and $(U^*, .)$ be two Lie-extendible left symmetric products. Then,  $(U\oplus U^*,\langle.,.\rangle_0,K_0)$, endowed with the Lie algebra bracket associated with the product given by $(1)$ is a para-K\"ahler Lie algebra.
			Where $\om_{0},\; \langle.,.\rangle_0$ and $K_0$ are  given by:
			\[\om_{0}(u+\al,v+\beta)=\beta(u)-\al(v),\quad \langle u+\al,v+\be\rangle_0 = \al(v)+\be(u)\esp K_0(u+\al) =u-\al.\]  Moreover, all para-K\"ahler Lie algebras are obtained in this manner.
		\end{theo}
		Let now $U$ be a 2-dimensional vector space and $U^*$  its dual space and let  $\{e_1,e_2\}$ , $\{e_3,e_4\}$ be a basis of $U$ and $U^*$.  We base  on the previous theorem and the classification of real left-symmetric algebras in dimension 2 listed below (see  Theorem 1.2. of \cite{K}),
		\begin{longtable}{l|l}
			
			$\mathfrak{b}_{1,\al}$ : $e_2.e_1=e_1,\,e_2.e_2=\al.e_2$ & $\mathfrak{b}_4$ : $e_1.e_2=e_1,\,e_2.e_2=e_1.e_2$ \\
			$\mathfrak{b}_2$   : $e_2.e_1=e_1,\,e_2.e_2= e_1 +e_2$&  $\mathfrak{b}_5^+$ :  $e_1.e_1=e_2,\,e_2.e_1=-e_1,\,e_2.e_2=-e_2$\\
			$\mathfrak{b}_{3,\al}$ $\al\not=0$: $e_1.e_2=e_1,\,e_2.e_1=(1-\frac{1}{\al})e_1,\,e_2.e_2=e_2$&$\mathfrak{b}_5^-$ : $e_1.e_1=-e_2,\,e_2.e_1=-e_1,\,e_2.e_2=-e_2$\\
			&\\
			$\mathfrak{c}_1$ : Trivial left-symmetric algebra & $\mathfrak{c}_2$ : $e_2.e_2=e_2$\\
			$\mathfrak{c}_5^+$ : $e_2.e_2=e_2,\,e_2.e_1=e_1,\,e_1.e_2=e_1,\,e_1.e_1=e_2$ &$\mathfrak{c}_3$ : $e_2.e_2=e_1$  \\
			$\mathfrak{c}_5^-$ : $e_2.e_2=e_2,\,e_2.e_1=e_1,\,e_1.e_2=e_1,\,e_1.e_1=e_2$&$\mathfrak{c}_4$ : $e_2.e_2=e_2,\,e_2.e_1=e_1,\,e_1
			e_2=e_1$.\\
		\end{longtable}
		\begin{remark}
			$\mathfrak{b}$ stands for algebras with non-commutative associated Lie algebra and $\mathfrak{c}$ stands  for algebras with commutative associated Lie algebra  .
		\end{remark}
		
		\begin{theo}\label{main2}Let $(\g,\langle.,.\rangle,K)$ be a four-dimensional   para-K\"ahler Lie algebra.
			Then there exists a basis $\{e_1,e_2,e_3,e_4\}$ of $\g$ such that
			\[\om=e^{13}+e^{24},\quad\langle.,.\rangle=\e^{13}+\e^{24}  \esp K=E_{11}+E_{22}-E_{33}-E_{44}\]
			and the non vanishing Lie brackets as listed in the Table \ref{tab4} and  \ref{tab5}.
		\end{theo}
		
		\begin{proof}
			We will give the proof in the case $\B_2$ since all cases should be handled in a similar way.
			In that case the left-symmetric  product in $U$ is given by $e_2.e_1=e_1$, $e_2.e_2=e_1+e_2$ and
			let
			\begin{eqnarray*}
				e_3.e_3 &=& a_{33} e_3+ b_{33} e_4 \\
				e_3.e_4&=& a_{34} e_3+ b_{34} e_4\\
				e_4.e_3 &=& a_{43} e_3+ b_{43} e_4 \\
				e_4.e_4 &=&a_{44} e_3+ b_{44} e_4
			\end{eqnarray*}
			be an arbitrary product in $U^*$, let's look for products in $U^*$ which  satisfy the Jacobi identity $\oint[[e_i,e_j],e_k]=0$ with $1\leq i<j<k\leq 4$,  where $\oint$ denotes the cyclic sum.
			
			The  identity  $\oint[[e_1,e_2],e_3]=0$  and $\oint[[e_1,e_2],e_4]=0$ is equivalent to
			\[ \begin{cases}
				b_{34}+a_{33}+a_{43}=0\\
				a_{44}=0\\
				b_{44}+a_{43}=0\\
			\end{cases}\]
			
			suppose that $a_{44}=0$, $b_{44}+a_{43}=0$ and $b_{34}=-a_{33}-a_{43}$, the  identity $\oint[[e_1,e_3],e_4]=0$  and $\oint[[e_2,e_3],e_4]=0$ is equivalent to
			
			\[\begin{cases}
				a_{33}a_{34}-2a_{33}a_{43}-a_{34}d_{43}-a_{43}^2-a_{43}d_{43}=0\\
				a_{34}(a_{34}+a_{43})=0\\
				a_{34}=0\\
			\end{cases} \esp \begin{cases}
				(a_{34}-3a_{43})d_{33}+d_{43}(a_{33}-d_{43})=0\\
				2a_{43}^2+(2a_{33}-a_{34}+d_{43})a_{43}-a_{34}(a_{33}-d_{43})=0\\
				a_{34}=0\\
				a_{34}+d_{43}+2a_{33}=0\\
			\end{cases}\]
			we get $a_{34}= 0$, $a_{43}= 0$, $a_{33}= 0$,   and  $d_{43} = 0$. Then the product in $U^*$ is given by $e_3.e_4=e_4$
			(who is indeed a left-symmetric  product) and the Lie bracket in $U\oplus U^*$ is given by
			
			\[[e_{1},e_{2}]=-e_{1}, \, [e_{2},e_{3}]=xe_{1}-e_{3}-e_{4}, \, [e_{2},e_{4}]=-e_{4}.\]
		\end{proof}

		\begin{proof}\textit{of the Theorem \ref{1-1}}.
			
			The Theorem \ref{main2} confirms that for each Lie algebra $\G$  of the tables \ref{tab4} and \ref{tab5} there exist a base $B_0=(e_1,e_2,e_3,e_4)$ such that the  para-k\"ahler structure is given by
			\[\om=e^{13}+e^{24}  \esp K=E_{11}+E_{22}-E_{33}-E_{44}\]
			ant the Lie brackets depend on some parameters. 
			In Tables \ref{tab6} and \ref{tab7} we build a family of isomorphisms (depending on the values of parameters) from $\G$ ($\mathfrak{B_{i,j}} \; or \; \mathcal{C}_{i,j}$) onto  a four-dimensional Lie algebra,  (say $A$) of the Table \ref{tab1}.
			Each isomorphism is given by the passage matrix $P$ from $B_0$ to $B= ( f_1, f_2, f_3, f_4)$. The image by $P$ of the para-K\"ahler structure $(\omega,K)$  is given by the matrices of its component in the bases $B$ and $B^*$ by
			\[^tP\circ\omega\circ P=\omega_i\quad and \quad P^{-1}\circ K\circ P=K_i.\]
			In this way we collect all the possible para-k\"ahler structures $(\omega_i,K_i)$ on $A$. Thereafter, we proceed
			to the classification in $A$ (up to automorphism).
			
			We will give the proof in the case $\mathfrak{r} \mathfrak{r}_{3,0}$ since all cases should be handled in a similar way. We will show that the Lie algebra $\mathfrak{r} \mathfrak{r}_{3,0}$ admits two non-equivalent para-K\"ahler structures. Note  that in this case  the non vanishing Lie bracket is
			\[[f_1,f_2]=f_2\] 
			the symplectic form is $\omega_0=f^{12}+f^{34}$ and the automorphisms is
			\[T=\left( \begin {array}{cccc}
			1&0&0&0\\ 
			a_{2,1}&a_{2,2}&0&0\\ a_{3,1}&0&a_{3,3}&a_{3,4}\\ a_{4,1}&0&a_{4,3}&a_{4,4}\end {array}
			\right).\] 
		The groups of automorphisms of four dimensional Lie algebras were
		given in \cite{O}.	
			
			From Table \ref{tab6} and Table \ref{tab7}, $\mathfrak{r} \mathfrak{r}_ {3.0} $ is obtained four times.
			\begin{enumerate}
				\item The transformation:  $f_1=-e_4, \;f_2=e_2, \;f_3=e_3, \; f_4=e_1 $ gives an isomorphism from 	$\mathcal{C}_{1,6}$
				to $\mathfrak{r} \mathfrak{r}_ {3,0}$ and the   para-K\"ahler structure obtained on $\mathfrak{r} \mathfrak{r}_ {3,0}$ is
				\[\omega_1=f^{12}-f^{34}\quad and\quad K_1=-E_{11}+E_{22}-E_{33}+E_{44}.\]
				\item The transformation: $f_1=-e_2, \;f_2=-ye_2+e_4, \;f_3=e_1, \; f_4=e_3$	
				gives an isomorphism from $\mathcal{C}_{2,1}$ with $x=0$  to $\mathfrak{r} \mathfrak{r}_ {3,0}$ and the   para-K\"ahler structure obtained on $\mathfrak{r} \mathfrak{r}_ {3,0}$ is
				\[\omega_2=-f^{12}+f^{34}\quad and\quad K_2=E_{11}-2yE_{12}-E_{22}+E_{33}-E_{44}.\]
				\item The transformation:  $f_1=-e_2, \;f_2=e_4, \;f_3=e_1, \; f_4=e_3 $ gives an isomorphism from 
				$\mathcal{C}_{2,2}$ with $ x=0,\; y=0$   to $\mathfrak{r} \mathfrak{r}_ {3,0}$ and the   para-K\"ahler structure obtained on $\mathfrak{r} \mathfrak{r}_ {3,0}$ is
				\[\omega_3=-f^{12}+f^{34}\quad and\quad K_3=E_{11}-E_{22}+E_{33}-E_{44}.\]
				\item The transformation:
				$f_1=e_1-e_2, \;f_2=e_4, \;f_3=e_1 ,\; f_4=e_3 $  gives an isomorphism from
				$\mathcal{C}_{2,3} $ with $ x=0,\; y=0$
				to $\mathfrak{r} \mathfrak{r}_ {3,0}$ and the   para-K\"ahler structure obtained on $\mathfrak{r}\mathfrak{r}_{3,0}$ is
				\[\omega_4=-f^{12}+f^{24}+f^{34}\quad and\quad K_4=E_{11}-E_{22}+E_{33}-E_{44}.\]
			\end{enumerate}
			the algebra $\mathfrak{r}\mathfrak{r}_{3,0}$ support $\omega_0$ as a unique symplectic structure (up to automorphism), therefore there are four families of automorphisms $T_i$, $i\in\{1,...,4\}$ such that,$T^*_i\omega_i=\omega_0$ for $i\in\{1,...,4\}$, a direct calculation gives us
			\[T_1=\left( \begin {array}{cccc} 1&0&0&0\\ a_{2,1}&1&0&0\\
			0&0&\frac{a_{3,4}a_{4,3}-1}{a_{4,4}}&a_{3,4}\\ 0&0&a_{4,3}&a_{4,4}\end {array}
			\right),\;T_2=\left( \begin {array}{cccc} 1&0&0&0\\ a_{2,1}&-1&0&0\\
			0&0&\frac{a_{3,4}a_{4,3}+1}{a_{4,4}}&a_{3,4}\\ 0&0&a_{4,3}&a_{4,4}\end {array}
			\right)
			\] 
			\[T_3=\left( \begin {array}{cccc} 1&0&0&0\\ a_{2,1}&1&0&0\\
			0&0&\frac{a_{3,4}a_{4,3}-1}{a_{4,4}}&a_{3,4}\\ 0&0&a_{4,3}&a_{4,4}\end {array}
			\right),\;T_4=\left( \begin {array}{cccc} 1&0&0&0\\ a_{2,1}&-1&0&0\\
			-1&0&\frac{a_{3,4}a_{4,3}+1}{a_{4,4}}&a_{3,4}\\ 0&0&a_{4,3}&a_{4,4}\end {array}
			\right).
			\] 
			Thus we obtain four  para-K\"ahler structures on $\mathfrak{r}\mathfrak{r}_{3,0}$ given by $(\omega_0,K_{0i})$, $i\in{1,...,4}$ with $K_{0i}=T^{-1}_i\circ K_i\circ T_i$  a direct calculation gives us
			\begin{itemize}
				\item[]$K_{01}=-E_{11}+E_{22}-E_{33}+E_{44}$
				\item[]$K_{02}=E_{11}+2yE_{12}-E_{22}+E_{33}-E_{44}$
				\item[]$K_{03}=E_{11}-E_{22}+E_{33}-E_{44}$
				\item[]$K_{04}=-E_{11}+E_{22}+E_{33}-E_{44}$
			\end{itemize} 
			Noticing that $K_{03}$ is a sub-case of $K_{02}$ and that $(\omega_0,K_{04})$ is isomorphic to $(\omega_0,K_{01})$. Indeed we have $L^* \omega_0=\omega_0$ and $L^{-1}_i\circ K_{04}\circ L_i=K_{01}$ 
			with 
			\[L=\left(\begin {array}{cccc} 
			1&0&0&0\\ 0&1&0&0
			\\ 0&0&0&1\\ 0&0&-1&0\end {array}
			\right).
			\]
			We complete the proof by showing that $(\omega_0,K_{01})$ is not isomorphic to $(\omega_0,K_{02})$. Indeed, the symplectomorphism group of $\omega_0$ is generated by
			\[L_1=\left( \begin {array}{cccc} 1&0&0&0\\ a_{{2,1}}&1&0&0\\ 0&0&{\frac {a_{{3,4}}a_{{4,3}}+1}{a_{{4,4}}}}&a
			_{{3,4}}\\ 0&0&a_{{4,3}}&a_{{4,4}}\end {array}
			\right)
			\quad and \quad L_2=\left( \begin {array}{cccc} 1&0&0&0\\ \noalign{\medskip}a_{{2,1}}&1&0
			&0\\ \noalign{\medskip}0&0&a_{{3,3}}&a_{{3,4}}\\ \noalign{\medskip}0&0
			&-{a_{{3,4}}}^{-1}&0\end {array} \right) 
			\]
			a simple calculation gives us \[f^2((L_1^{-1}\circ K_{01}\circ L_1-K_{02})(f_2))=2\quad and \quad f^1((L_2^{-1}\circ K_{01}\circ L_2-K_{02})(f_1))=-2\] so $L_1^{-1}\circ K_{01}\circ L_1\not=K_{02}$ and $L_2^{-1}\circ K_{01}\circ L_2\not=K_{02}$. 
		\end{proof}
		\section{Application:  Curvature properties of four-dimensional  para-K\"ahler Lie algebras}

		Let now $(\G,\omega,K)$ denote a four-dimensional para-K\"ahler Lie algebra.
		Let  $\nabla : \G\times\G\longrightarrow\G$ be the Levi-Civita product associated to a left-invariant pseudo-Riemannian metric $h(X,Y)=\omega(KX,Y)$. The connection $\nabla$ is also called  Hess connection. The curvature tensor is then described in terms of the map
		\begin{equation}
			\begin{array}{rcl}
				R :\quad  \G\times\G& \longrightarrow&gl(\G) \\
				(X,Y) & \longmapsto&  R(X,Y)=\nabla_{[X,Y]}-[\nabla_X, \nabla_Y]
			\end{array}.
		\end{equation}
		The Ricci tensor is the symmetric tensor $ric$ given by
		$ric(X,Y) = tr(Z\longmapsto R(X,Z)Y)$
		and the Ricci operator $Ric : \G\longrightarrow\G$ is given by the relation $h(Ric(X),Y) = ric(X,Y).$ The scalar curvature is  defined in the standard way by  $s=tr(Ric)$.

		Recall that: $(\G,h)$  is called flat if $R= 0$, Ricci flat if $Ric= 0$ and Ricci soliton if
		\begin{equation}
			\mathcal{L}_X h +ric=\lambda h,
		\end{equation}
		where  $X=x_1e_1+x_2e_2+x_3e_3+x_4e_4$ is a vector field and $\lambda$ is a real constant, in that case if $X=0$ then $h$ is called Einstein metric and if  $\lambda$ is positive, zero, or negative then $h$ is called a shrinking, steady, or expanding Ricci soliton, respectively. We give in the following theorem some geometrical situations for the left
		invariant four-dimensional dimensional  para-K\"ahler 
		Lie groups.
		
		\begin{theo}			
			Let $(\G,\omega,K)$ be a class of para-K\"ahler Lie algebras obtained in Theorem \ref{1-1}. The associated para-K\"ahler metric and some of his properties are given in the following tables
		\begin{longtable}{clcccc}
					\hline
					\multirow{2}{*}{Lie algebra}   &\multirow{2}{*}{  Para-K\"ahler metric}   & \multirow{2}{*}{R=0} &\multirow{2}{*}{Ric=0}    &\multicolumn{2}{c}{Ricci soliton}  \\
					&&&&$\lambda$&$X$\\
					\hline
					\multirow{2}{*}{$\mathfrak{rh}_{3}$}&$\mp(\e^{14}-\e^{23})$&Yes&Yes&0&$(0,0,x_3,x_4)$\\
					&$\e^{13}-\e^{14}+\e^{23}$&Yes&Yes&0&$(0,0,x_3,x_4)$\\
					\hline
					\multirow{2}{*}{$ \mathfrak{rr}_{3,0}$}&$\pm(\e^{12}+\e^{34})$&Yes&Yes&0&$(0,0,x_3,x_4)$\\
					&$\e^{12}+x\e^{22}+\e^{34}$, $x\not=0$&No&No&&No\\
					\hline
					\multirow{2}{*}{$ \mathfrak{rr}_{3,-1}$}&$\pm(\e^{14}+\e^{23}\pm\e^{33})$&No&Yes&0&$(0,0,0,x_4)$\\
					&$\e^{14}+\e^{23}+x\e^{44}$&Yes&Yes&0&$(x_1,0,0,x_4)$\\
					&$\e^{14}\pm \e^{22}-\e^{23}$&No&Yes&$0$&$(0,0,0,x_{4})$\\
					&$\e^{14}-\e^{23}$&Yes&Yes&0&$(x_{1},0,0,x_{4})$\\
					\hline
					$ \mathfrak{r}_{2}\mathfrak{r}_{2}$ &$\pm(-\e^{12}-\mu\e^{13}+2\e^{23}+2\mu\e^{33}+\e^{34})$&No&Yes&$0$&$(0,0,0,0)$\\
					$\mu>0$&$-\e^{12}-\mu\e^{13}+\e^{34}$&No&Yes&$0$&$(0,0,0,0)$\\
					\hline
					&$\e^{12}+2\e^{14}-\e^{34}$&Yes&Yes&$-x_3$&$(x_3,0,x_3,0)$\\
					&$-\e^{12}+x\e^{22}+\e^{23}+x\e^{24}-\frac1x\e^{33}+x\e^{44}$&No&No&$\frac{3}{2}x$&$(0,0,0,0)$\\
					&$-\e^{12}-\e^{34}$&Yes&Yes&$-x_3$&$(x_3,0,x_3,0)$\\
					$\mathfrak{r}_{2}\mathfrak{r}_{2}$&$-\e^{12}-x\e^{33}+\e^{34}$&Yes&Yes&$-x_1$&$(x_1,0,x_1,\frac{x}{2}x_1)$\\
					$\mu=0$&$\e^{12}-\e^{23}-\e^{34}$&Yes&Yes&$-x_3$&$(x_3,0,x_3,0)$\\
					&$\e^{12}+x\e^{22}+\e^{34}+y\e^{44}$, $xy\not=0$, $x\not=y$ &No&No&&No\\
					&$\e^{12}+x\e^{22}+\e^{34}+x\e^{44}$, $x\not=0$ &No&No&$x$&$(0,0,0,0)$\\
					&$\e^{12}+x\e^{22}+\e^{34}$, $x\not=0$&No&No&$-x$&$(0,0,x,0)$\\
					&$\e^{12}+\e^{34}+y\e^{44}$, $y\not=0$ &No&No&$y$&$(-y,0,0,0)$\\
					&$\e^{12}+\e^{34}$&Yes&Yes&$-x_3$&$(x_3,0,x_3,0)$\\
					&$\e^{12}+x\e^{22}+x\e^{24}+\e^{34}+x\e^{44}$, $x\not=0$&No&No&$\frac{3}{2}x$&$(0,0,0,0)$\\
					\hline
					&$\e^{14}+\frac4x\e^{22}-\e^{23}+x\e^{44}$&No&No&$\frac32x$&$(0,0,0,0)$\\
					&$-\e^{14}-\e^{23}$&Yes&Yes&$-x_1$&$(x_1,0,0,0)$\\
					&$h_1$,\qquad $xy\not=0$&No&No&&No\\
					$\mathfrak{r}_{2}'$&$h_2$,\qquad $y\not=0$&No&No&$-2y$&$(0,0,0,0)$\\
					&$-(2+x)\e^{12}+(x+1)(\e^{14}+\e^{23})-x\e^{34}$, $x\not=0$&No&No&&No\\
					&$-2\e^{12}+\e^{14}+\e^{23}$&Yes&Yes&$-x_3$&$(x_3,0,x_3,0)$\\
					&$\e^{14}+\e^{23}+\e^{33}+\e^{41}+x\e^{44}$, $x\not=0$&No&No&$\frac32x$&$(-\frac32x,0,0,0)$\\
					&$\e^{14}+\e^{23}+\e^{33}+\e^{41}$&Yes&Yes&$-x_1$&$(x_1,0,0,0)$\\
					\hline
					
					$\mathfrak{r}_{4,0}$ &$\pm(\e^{14}\pm\e^{23})$&No&Yes&$0$&$(0,x_2,0,0)$\\
					\hline
					\multirow{2}{*}{$\mathfrak{r}_{4,-1}$} &$\e^{13}-\e^{24}+x\e^{33}$, $x\not=0$&No&Yes&$0$&$(0,0,0,0)$\\
					&$\pm(\e^{13}-\e^{24})$&Yes&Yes&$0$&$(0,0,0,0)$\\
					\hline
					\multirow{6}{*}{$\mathfrak{r}_{4,-1,\beta}$} &$\mp(\e^{12}\pm\e^{22}-\e^{34})$,\quad$\beta\not=0$&No&Yes&$0$&$(0,0,0,0)$\\
					&$\e^{12}\mp\e^{34}$,\qquad\qquad $\beta\not=0$&Yes&Yes&$0$&$(0,0,0,0)$\\
					&$\e^{12}-\frac1x\e^{33}+x\e^{44}$,\quad $\beta\not=0$&No&No&&No\\
					&$\mp(\e^{12}\pm\e^{22}-\e^{34})$,\quad$\beta=0$&No&Yes&$0$&$(0,0,x_3,0)$\\
					&$\e^{12}\mp\e^{34}$,\qquad\qquad $\beta=0$&Yes&Yes&$0$&$(0,0,x_3,x_4)$\\
					&$\e^{12}-\frac1x\e^{33}+x\e^{44}$,\quad $\beta=0$&No&Yes&$0$&$(0,0,x_3,x_4)$\\
					\hline
					\multirow{6}{*}{$\mathfrak{r}_{4,-1,-1}$} &$-x\e^{11}-\e^{12}-\e^{13}-\e^{34}$,\quad$x\not=0$&No&Yes&$0$&$(0,0,0,0)$\\
					&$-\e^{12}-\e^{13}-\e^{34}$&Yes&Yes&$0$&$(0,0,0,0)$\\
					&$\mp(\mp\e^{11}-\e^{12}-\e^{34})$&No&Yes&$0$&$(0,0,0,0)$\\
					&$-\e^{12}\mp\e^{34}$&Yes&Yes&$0$&$(0,0,0,0)$\\
					&$\e^{12}\mp\e^{22}-\e^{34}$&No&Yes&$0$&$(0,0,0,0)$\\
					&$\pm\e^{12}+ x\e^{33}-\e^{34}$,\quad$x\not=0$&No&No&$0$&$(0,0,0,0)$\\
					&$\e^{12}-\e^{34}$&Yes&Yes&$0$&$(0,0,0,0)$\\
										\hline
					\multirow{4}{*}{$\mathfrak{r}_{4,\alpha,-\alpha}$ }&$\mp(\e^{14}-\e^{23}\mp\e^{33})$&No&Yes&$0$&$(0,0,0,0)$\\
					&$\mp(x\e^{11}-\e^{14}+\e^{23})$,\qquad$x\not=0$&No&No&&No\\
					&$\mp(\e^{14}\mp\e^{23})$&Yes&Yes&$0$&$(0,0,0,0)$\\
					&$-\e^{14}\mp\e^{22}-\e^{23}$&No&Yes&$0$&$(0,0,0,0)$\\
					\hline
					&$\mp(\e^{12}\mp\e^{22}+\e^{34})$&Yes&Yes&$0$&$(x_1,0,0,0)$\\
					&$\e^{12}+x\e^{33}+\e^{34}$,$x\not=0$&No&Yes&$\frac{3}{2}x$&$(0,0,0,0)$\\
					&$\e^{12}+\e^{34}$&Yes&Yes&$-x_4$&$(x_1,0,0,x_4)$\\
					&$\mp(\e^{12}-\e^{34})$&Yes&Yes&$-x_4$&$(0,0,0,x_4)$\\
					$ \mathfrak{d}_{4,1}$&$-\e^{11}+\e^{12}-x\e^{13}-\e^{34}$, $x\not=0$&No&No&&No\\
					&$\mp(\e^{11}-\e^{12}+\e^{34})$&No&Yes&$0$&$(0,0,0,0)$\\
					&$-\e^{11}-\e^{12}+\e^{34}$&No&Yes&$0$&$(0,0,0,0)$\\
					&$-\e^{12}+x\e^{33}+\e^{34}$, $x\not=0$&No&No&$\frac{3}{2}x$&$(0,0,0,0)$\\
					&$\pm(\e^{12}+x\e^{22}-\e^{24}+\e^{34})$&Yes&Yes&$0$&$(x_1,0,0,0)$\\
					
					\hline
					&$\e^{12}\mp\e^{22}+\e^{34}$&No&Yes&$0$&$(\mp x_1,x_1,0,0)$\\
					&$\e^{12}+\e^{34}$&Yes&Yes&$-x_4$&$(0,x_2,0,x_4)$\\
					&$\e^{12}+x\e^{33}-\e^{34}$, $x\not=0$&No&No&$\frac32x$&$(0,0,0,0)$\\
					&$\e^{12}-\e^{34}$&Yes&Yes&$-x_4$&$(0,0,0,x_4)$\\
					&$2x\e^{11}-\e^{14}\mp\frac1x\e^{22}\mp\e^{23}$&No&No&&No\\
					$\mathfrak{d}_{4,2}$    &$-2\e^{12}\mp\e^{14}+x\e^{22}\pm\e^{23}$&No&No&&No\\
					&$\e^{14}+\e^{23}+x\e^{33}$, &No&No&$0$&$(0,0,0,0)$\\
					&$\mp(\e^{14}+\e^{23})$ &Yes&Yes&$0$&$(0,0,0,0)$\\
					&$\mp(\e^{14}-\e^{23})$ &Yes&Yes&$0$&$(0,0,0,0)$\\
					&$-x\e^{11}-\e^{14}\mp\e^{23}$, $x\not=0$ &No&No&$0$&$(0,0,0,0)$\\ 
					&$\mp x\e^{11}+\e^{14}\pm\e^{23}+2x\e^{33}$, $x\not=0$ &No&No&$0$&$(0,0,0,0)$\\ 
					&$2x\e^{11}-2x\e^{13}-\e^{14}-\e^{23}+2x\e^{33}$, $x\not=0$ &No&No&$0$&$(0,0,0,0)$\\
					\hline
					\multirow{2}{*}{ $\mathfrak{d}_{4,\frac{1}{2}}$}&$\e^{12}+x\e^{33}-\e^{34}$, $x\not=0$&No&No&$\frac32x$&$(0,0,0,0)$\\
					&$\e^{12}\mp\e^{34}$&Yes&Yes&$-x_4$&$(0,0,0,x_4)$\\
					\hline
					\multirow{5}{*}{$\mathfrak{d}_{4,\lambda}$}&$\e^{12}\mp\e^{22}+\e^{34}$&No&Yes&$0$&$(0,0,0,0)$\\
					&$\mp\e^{11}-\e^{12}+\e^{34}$&No&Yes&$0$&$(0,0,0,0)$\\
					&$\mp\e^{12}+x\e^{33}+\e^{34}$, $x\not=0$&No&No&$\frac32x$&$(0,0,0,0)$\\
					&$\mp\e^{12}+x\e^{22}-\e^{34}$, $x\not=0$&No&Yes&$0$&$(0,0,0,0)$\\
					&$\mp(\e^{12}\mp\e^{34})$&Yes&Yes&$-x_4$&$(0,0,0,x_4)$\\
					\hline
					$ \mathfrak{h}_{4}$ &$\pm(\e^{12}-\e^{34})$&No&Yes&$0$&$(0,0,0,0)$\\
					\hline
					\caption{Curvature properties of four-dimensional  para-K\"ahler Lie algebras}
				\end{longtable}
						$h_1=y(\e^{11}-\e^{13}-\e^{22}+\e^{24}+\e^{33}-\e^{44})-(2+x)\e^{12}+(x+1)(\e^{14}+\e^{23})-x\e^{34}$\\
			$h_2=y(\e^{11}-\e^{13}-\e^{22}+\e^{24}+\e^{33}-\e^{44})-2\e^{12}+\e^{14}+\e^{23}$
		\end{theo}
		\begin{proof}
			We report below the details for the case of  $\mathfrak{d}_{4,\frac{1}{2}}$ the other cases are treated in the same way. Let $\{e_1, e_2, e_3, e_4\}$ denotes the basis used in theorem \ref{1-1} for $\mathfrak{d}_{4,\frac{1}{2}}$. The non isomorphic  para-K\"ahler structures in $\mathfrak{d}_{4,\frac{1}{2}}$ are 
			$(\omega,K_1)$ and  $(\omega ,K_2)$ with
			$\omega =e^{12}-e^{34}$,
			$K_1=E_{11}-E_{22}-E_{33}+xE_{43}+E_{44}$ and $K_2=E_{11}-E_{22}+E_{33}-E_{44}$.
			
			The corresponding compatible  metric to $(\omega,K_i)$ is uniquely determined by $h_i(X, Y) = (K_iX, Y)$. Hence, para-K\"ahler metrics in $\mathfrak{d}_{4,\frac{1}{2}}$ are of the form
			
			\[h_1=\left(\begin {array}{cccc} 0&1&0&0\\1&0&0&0
			\\ 0&0&x&1\\ 0&0&1&0\end {array}
			\right)\quad x\in\R \qquad and \qquad h_2=\left( \begin {array}{cccc} 0&1&0&0\\ 1&0&0&0
			\\ 0&0&0&-1\\ 0&0&-1&0\end {array}
			\right).
			\]
			For $h_1$, $x\not=0$,
			using the Koszul formula, the Levi-Civita connection  is described
			
			\[\nabla_{e_1}=\left(\begin {array}{cccc}
			0&0&-\frac12x&-1\\
			0&0&0&0\\
			0&1&0&0\\
			0&-\frac12x&0&0\end {array}
			\right), \qquad  \qquad\nabla_{e_2}=\left( \begin {array}{cccc}
			0&0&0&0\\
			0&0&\frac12x&0\\
			0&0&0&0\\
			-\frac12x&0&0&0\end {array}
			\right).
			\]
			\[\nabla_{e_3}=\left(\begin {array}{cccc}
			-\frac12x&0&0&0\\
			0&\frac12x&0&0\\
			0&0&x&0\\
			0&0&-x^2&-x\end {array}
			\right), \qquad  \qquad\nabla_{e_4}=\left( \begin {array}{cccc}
			-\frac12&0&0&0\\
			0&\frac12&0&0\\
			0&0&1&0\\
			0&0&-x&-1\end {array}
			\right).
			\]
			Then we  calculate the curvature matrices $R(e_i,e_j)$ (for $1\leq i<j\leq 4$) and we find
			\[R(e_1,e_2)=\left( \begin {array}{cccc} -x&0&0&0\\ 0&x&0&0\\ 0&0&\frac{x}{2}&0\\ 
			0&0&-\frac12x^{2}&-\frac12x\end {array} \right),\qquad R(e_1,e_3)=\left(\begin {array}{cccc}
			0&0&-\frac14\,{x}^{2}&-\frac{x}{2}
			\\0&0&0&0\\ 0&\frac{x}{2}&0&0
			\\ 0&-\frac14 {x}^{2}&0&0\end {array} \right)
			\]
			\[R(e_2,e_3)=\left( \begin {array}{cccc} 0&0&0&0\\ \noalign{\medskip}0&0&-\frac14{x}
			^{2}&0\\ \noalign{\medskip}0&0&0&0\\\frac14{x}^{2}&0&0
			&0\end {array} \right),\qquad R(e_2,e_4)=\left(\begin {array}{cccc} 0&0&0&0\\ 0&0&-\frac x2&0
			\\0&0&0&0\\ \frac x2&0&0&0\end {array}
			\right)
			\]
			\[R(e_3,e_4)=\left(\begin {array}{cccc} \frac x2&0&0&0\\ 0&-\frac x2&0&0
			\\ 0&0&-x&0\\0&0&x^2&x
			\end {array} \right)\qquad and\quad R(e_1,e_4)=0. 
			\]
			The Ricci tensor $ric$ and the Ricci operator $Ric$ are given
			by
			\[ric=\left( \begin {array}{cccc} 0&\frac32x&0&0\\ \frac32x&0
			&0&0\\ 0&0&\frac32{x}^{2}&\frac32x\\ \noalign{\medskip}0
			&0&\frac32x&0\end {array} \right)
			\qquad and \qquad Ric= \left( \begin {array}{cccc} \frac32x&0&0&0\\ \noalign{\medskip}0&\frac32x
			&0&0\\0&0&\frac32x&0\\ 0&0&0&\frac32x
			\end {array} \right) 
			\]
			The Lie derivative $\mathcal{L}_Xh_1$ of the metric $h_1$ with respect to an arbitrary vector field $X=x_1e_1+x_2e_2+x_3e_3+x_4e_4\in\G$ is given by
			\[\mathcal{L}_Xh_1= \left( \begin {array}{cccc} 0&-x_4&x_2x&\frac32x_2
			\\-x_4&0&-x_1x&-\frac12x_1
			\\ x_2x&x_1x&-2x_4x&xx_3-x_4
			\\ \frac 32x_2&-\frac12x_1&xx_3-x_4&2x_3\end {array} \right).\]
			Then, solving equation $\mathcal{L}_X h +ric=\lambda h$, for $x\not=0$ we obtain
			\[\lambda=\frac32x\qquad and\qquad X=0.\]
			Notice that, in this case, the   para-K\"ahler metric is a Einstein metric not Ricci flat.
			
			For $h_1$ with $x=0$ and $h_2$	
			\[\nabla_{e_1}e_2=e_3,\;\nabla_{e_1}e_4=-e_1,\;\nabla_{e_4}e_1=-\frac12e_1,\;\nabla_{e_4}e_2=\frac12e_2,\;\nabla_{e_4}e_3=e_3,\;\nabla_{e_4}e_4=-e_4\]
			This para-K\"ahler structure  is flat 	($R(e_i,e_j)=0$ for $1\leq i<j\leq 4$). The Lie derivative $\mathcal{L}_Xh_1$ of the metric $h_1$,  is given by
			\[\mathcal{L}_Xh_1=\left(\begin {array}{cccc} 0&-x_{{4}}&0&\frac32x_2
			\\-x_4&0&0&-\frac12x_1\\0&0
			&0&-x_4\\ \frac32x_2&-\frac12x_1&-x_4&2
			x_3\end {array} \right).
			\]
			Then, solving equation $\mathcal{L}_X h=\lambda h$, for $x=0$ (or for $h_2$) we obtain
			\[\lambda=-x_4\qquad and\qquad X=x_4e_4.\]
		\end{proof}		
		\section{Tables}
		{\renewcommand*{\arraystretch}{1.2}
			\begin{center}	\begin{longtable}{|l|l|}
					\hline
					Lie algebra &\multicolumn{1}{c|}{No zero brackets}\\
					\hline
					$\B_{1,\al}^1$ 
					\begin{tiny}
						$\al\not\in\{-2,-1,1\}$
					\end{tiny} &$[e_1,e_2]=-e_1,\, [e_2,e_3]=xe_1-e_3,\, [e_2,e_4]=-\al e_4$ \\
					\hline
					$\B_{1,\al}^2$
					\begin{tiny}
						$\al\not\in\{-2,-1,0,1\}$
					\end{tiny}
					&  $[e_1,e_2]=-e_1,\, [e_1,e_4]=-\frac{x}{\al} e_1,\, [e_2,e_3]=-e_3,\, [e_2,e_4]=xe_2-\al e_4$,\, $ [e_3,e_4]=\frac{x}{\al} e_3$\\
					\hline
					{$\B_{1,-2}^1$}& $ [e_1,e_2]=-e_1,\, [e_1,e_4]=xe_1,\, [e_2,e_3]=ye_1-e_3,\;[e_2,e_4]=2xe_2+2e_4,\;[e_3,e_4]=-xe_3$\\
					\hline
					$\B_{1,-2}^2$&$ [e_1,e_2]=-e_1,\, [e_1,e_3]=xe_1,\, [e_2,e_3]=ye_1-xe_2-e_3,\;[e_2,e_4]=xe_1+2e_4,\;[e_3,e_4]=-2xe_4$\\
					\hline
					$\B_{1,-1}^1$& $[e_1,e_2]=-e_1,\, [e_1,e_3]=-xe_1,\, [e_2,e_3]=ye_1+xe_2-e_3, \, [e_2,e_4]=e_4, \, [e_3,e_4]=xe_4$\\
					\hline
					$\B_{1,-1}^2$& $[e_1,e_2]=-e_1,\, [e_1,e_4]=xe_1,\, [e_2,e_3]=-e_3, \, [e_2,e_4]=xe_2+e_4, \, [e_3,e_4]=-xe_3$\\
					\hline
					$\B_{1,0}$&$[e_1,e_2]=-e_1, \, [e_1,e_4]=xe_{1}, \, [e_2,e_3]=-e_3, \, [e_3,e_4]=-xe_3$\\
					\hline
					\multirow{2}{*}{$\B_{1,1}^1$\; $x\not=0$}&$[e_1,e_2]=-e_1, \, [e_{1},e_{3}]=-\frac{y}{2}e_1, \, [e_1,e_4]=-xe_1, \,
					[e_2,e_3]=\frac{y^2}{2x}e_1+\frac{y}{2}e_{2}-e_{3}$\\& $[e_{2},e_{4}]=ye_{1}+xe_{2}-e_{4}, \, [e_{3},e_{4}]=xe_{3}-\frac{y}{2}e_{4}$\\
					\hline
					$\B_{1,1}^2$&$[e_1,e_2]=-e_1,\, [e_2,e_3]=xe_1-e_3, \, [e_2,e_4]=-e_4$\\
					\hline
					
					$\B_{2}$&$[e_{1},e_{2}]=-e_{1}, \, [e_{2},e_{3}]=xe_{1}-e_{3}-e_{4}, \, [e_{2},e_{4}]=-e_{4}$  \\
					\hline
					$\B_{3,\alpha}^1 \; \alpha\neq0$& $[e_1,e_2]=\frac{1}{\alpha} e_1,\, [e_1,e_3]=[e_2,e_4]=-e_4,\, [e_2,e_3]=xe_1+\frac{1-\alpha}{\alpha}e_3$\\
					\hline
					$\B_{3,\alpha}^2\; \alpha\neq0$& $[e_1,e_2]=\frac{1}{\alpha} e_1,\, [e_1,e_3]=[e_2,e_4]=x\alpha e_2-e_4,\, [e_1,e_4]=xe_1, \, [e_2,e_3]=\frac{1-\alpha}{\alpha}e_3, \, [e_3,e_4]=x(\alpha-1)e_3$\\
					\hline
					$\B_{3,\frac{1}{2}}^1$& $[e_1,e_2]=2e_1,\, [e_1,e_3]=x e_1-e_4, \, [e_2,e_3]=ye_1-\frac{x}{2}e_2+e_3, \,[e_2,e_4]=-e_4,\; [e_3,e_4]=-\frac{x}{2} e_4$\\
					\hline
					\multirow{2}{*}{$\B_{3,\frac{1}{2}}^2$ $y\neq0$}& $[e_1,e_2]=2e_1,\, [e_1,e_3]=-2x e_1+\frac{y}{2} e_2-e_4, \,[e_1,e_4]=ye_1,\; [e_2,e_3]=-\frac{3x^{2}}{y} e_1+\frac{x}{2}e_2+e_3$\\
					&$[e_2,e_4]=xe_1+\frac{y}{2}e_2-e_4,\; [e_3,e_4]=-\frac{y}{2} e_3-\frac{x}{2} e_4$\\
					\hline
					$\B_{3,\frac{1}{2}}^3$& $[e_1,e_2]=2e_1,\, [e_1,e_3]=[e_2,e_4]=x e_1-e_4, \, [e_2,e_3]=ye_1-xe_2+e_3, \, [e_3,e_4]=-2x e_4$\\
					\hline
					$\B_{3,1}$& $[e_1,e_2]=e_1,\, [e_1,e_3]=x e_1+ye_2-e_4, \,[e_1,e_4]=ye_1,\; [e_2,e_3]=ze_1, \, [e_2,e_4]=y e_2-e_4$\\
					\hline
					$\B_{4}$&$[e_{1},e_{2}]=e_{1}, \, [e_{1},e_{3}]=xe_{1}-e_{4}, \, [e_{2},e_{3}]=ye_{1}-e_{4}, \, [e_{2},e_{4}]=-e_{4}$\\
					\hline
					$\B_{5,1}^{+}$&$[e_{1},e_{2}]=e_{1}, \, [e_{1},e_{4}]=x e_{1}-e_3, \, [e_{2},e_{3}]=e_{3}, \, [e_{2},e_{4}]=-2xe_{2}+2e_4,\;[e_3,e_4]=-xe_3$\\
					\hline
					$\B_{5,2}^{+}$&$[e_{1},e_{2}]=e_{1}, \, [e_{1},e_{4}]=-\frac{x}{4} e_{1}-e_3, \, [e_{2},e_{3}]=xe_1+e_{3}, \, [e_{2},e_{4}]=-\frac{x}{2} e_{2}+2e_4,\;[e_3,e_4]=\frac{x}{4} e_3$\\
					\hline
					$\B_{5,3}^{+}$&$[e_{1},e_{2}]=e_{1}, \, [e_{1},e_{3}]=-x e_{1}, \, [e_{1},e_{4}]=-2xe_1-xe_2-e_{3}, \, [e_{2},e_{3}]=2xe_1+xe_2+e_{3}$\\&$[e_2,e_4]=3x e_1+2xe_2+2e_4,\;[e_3,e_4]=2xe_3-2xe_4$\\
					\hline
					$\B_{5,4}^{+}$&$[e_{1},e_{2}]=e_{1}, \, [e_{1},e_{3}]=-x e_{1}, \, [e_{1},e_{4}]=2xe_1-xe_2-e_{3}$\\&$ [e_{2},e_{3}]=-2xe_1+xe_2+e_{3},\;[e_2,e_4]=3x e_1-2xe_2+2e_4,\;[e_3,e_4]=-2xe_3-2xe_4$\\
					\hline
					$\B_{5,1}^{-}$&$[e_{1},e_{2}]=e_{1}, \, [e_{1},e_{4}]=x e_{1}+e_3, \,  [e_{2},e_{3}]=e_{3},\;[e_2,e_4]=-2x e_2+2e_4,\;[e_3,e_4]=-xe_3$\\
					\hline
					$\B_{5,2}^{-}$&$[e_{1},e_{2}]=e_{1}, \, [e_{1},e_{4}]=\frac{x}{4} e_{1}+e_3, \,  [e_{2},e_{3}]=xe_1+e_{3},\;[e_2,e_4]=\frac{x}{2} e_2+2e_4,\;[e_3,e_4]=-\frac{x}{4} e_3$\\
					\hline
					\caption{Four dimensional Para-K\"ahler Lie algebras coming from $\mathfrak{b}$}
					\label{tab4}
				\end{longtable}
		\end{center}}

		\newpage
		{\renewcommand*{\arraystretch}{1.4}
			\begin{center}	\begin{longtable}{|l|l|}
					\hline
					Lie algebra &\multicolumn{1}{c|}{No zero brackets}\\
					\hline
					$\mathcal{C}_{1,1}$&$[e_{1},e_{4}]=e_{1}, \, [e_2,e_4]=\alpha e_2 , \; [e_3,e_4]=-e_3$\\
					\hline
					$\mathcal{C}_{1,2}$&$[e_{1},e_{4}]=e_{1}+e_2, \, [e_2,e_4]= e_2 , \; [e_3,e_4]=-e_3$\\
					\hline
					$\mathcal{C}_{1,3}$&$[e_{1},e_{3}]=[e_{2},e_{4}]=e_{2}, \, [e_1,e_4]=(1-\frac{1}{\alpha})e_1, \; [e_{3},e_{4}]=\frac{1}{\alpha}e_3$\\
					\hline
					$\mathcal{C}_{1,4}$&$[e_{1},e_{3}]=[e_{1},e_{4}]=[e_{2},e_{4}]=e_{2}, \, [e_3,e_4]=e_3$\\
					\hline
					$\mathcal{C}_{1,5}^{+}$&$[e_{1},e_{4}]=-e_{1}, \; [e_{2},e_{3}]=e_1 , \, [e_2,e_4]=-2 e_2,\;[e_3,e_4]=e_3$\\
					\hline
					$\mathcal{C}_{1,5}^{-}$&$[e_{1},e_{4}]=[e_{2},e_{3}]=-e_{1}, \, [e_2,e_4]=-2 e_2,\;[e_3,e_4]=e_3$\\
					\hline
					$\mathcal{C}_{1,6}$&$[e_{2},e_{4}]=e_{2}$\\
					\hline
					$\mathcal{C}_{1,7}$&$[e_{1},e_{4}]=e_{2}$\\
					\hline
					$\mathcal{C}_{1,8}$&$[e_{1},e_{3}]=[e_{2},e_{4}]=e_{2}, \, [e_1,e_4]= e_1$\\
					\hline
					$\mathcal{C}_{1,9}$&$[e_{1},e_{3}]=[e_{2},e_{4}]=e_{2}, \, [e_1,e_4]=[e_2,e_3]= e_1$\\
					\hline
					$\mathcal{C}_{1,10}$&$[e_{1},e_{3}]=[e_{2},e_{4}]=e_{2}, \, [e_1,e_4]= e_1 , \; [e_2,e_3]=-e_1$\\
					\hline
					$\mathcal{C}_{2,1}$&$[e_{1},e_{3}]=xe_{1}, \, [e_2,e_4]=ye_2-e_4$\\
					\hline
					
					$\mathcal{C}_{2,2}$&$[e_{1},e_{3}]=xe_{1}, \,[e_2,e_3]=ye_1,\; [e_2,e_4]=-e_4$\\
					\hline
					$\mathcal{C}_{2,3}$&$[e_{1},e_{3}]=xe_{1}, \,[e_2,e_3]=ye_1,\; [e_2,e_4]=xe_1-e_4,\;[e_3,e_4]=-xe_4$\\
					\hline
					$\mathcal{C}_{3,1}$&$[e_{1},e_{3}]=xe_{1}, \,[e_2,e_3]=ye_1+ze_2-e_4,\; [e_3,e_4]=ze_4$\\
					\hline
					$\mathcal{C}_{3,2}$&$[e_{1},e_{3}]=[e_2,e_4]=xe_{1}, \,[e_2,e_3]=ye_1+ze_2-e_4,\; [e_3,e_4]=(z-x)e_4$\\
					\hline
					$\mathcal{C}_{4,1}$&$[e_{1},e_{3}]=[e_2,e_4]=xe_{1}-e_4, \,[e_2,e_3]=ye_1+xe_2-e_3$\\
					\hline
					$\mathcal{C}_{4,2}$&$[e_{1},e_{3}]=[e_2,e_4]=xe_{2}-e_4, \,[e_2,e_3]=-e_3,\; [e_3,e_4]=xe_3$\\
					\hline
					$\mathcal{C}_{5,1}^{+}$&$[e_{1},e_{3}]=[e_{2},e_{4}]=xe_{1}+ye_2-e_4, \,[e_1,e_4]=[e_{2},e_{3}]=ye_1+xe_2-e_3$\\
					\hline
					$\mathcal{C}_{5,2}^{+}$&$[e_{1},e_{3}]=[e_2,e_4]=xe_{2}-e_4, \,[e_1,e_4]=[e_2,e_3]=-e_3,\;[e_3,e_4]=xe_3$\\
					\hline
					$\mathcal{C}_{5,1}^{-}$&$[e_{1},e_{3}]=[e_2,e_4]=xe_{1}+ye_2-e_4, \,[e_1,e_4]=ye_1-xe_2+e_3,\;[e_2,e_3]=-ye_1+xe_2-e_3$\\
					\hline
					$\mathcal{C}_{5,2}^{-}$&$[e_{1},e_{3}]=[e_2,e_4]=xe_{2}-e_4, \,[e_1,e_4]=e_3,\;[e_2,e_3]=-e_3,\;[e_3,e_4]=xe_3$\\
					\hline
					\caption{Four dimensional Para-K\"ahler Lie algebras coming from $\mathfrak{c}$}
					\label{tab5}
				\end{longtable}
		\end{center}}

		{\renewcommand*{\arraystretch}{1.4}	
			\begin{center}
				
				\begin{longtable}{|ll|p{9.5cm}|c|}
					\hline
					Source && \multicolumn{1}{c|}{Isomorphism}&Target \\
					\hline
					$\B_{1,\al}^1$& $\vert\al\vert<1$&$f_1=e_1,\,f_2=-\frac{x}{2}e_1+e_3,\,f_3=e_4,\,f_4=e_2$& $\mathfrak{r}_{4,-1,-\al}$\\
					\hline
					$\B_{1,\al}^1$&$ \vert\al\vert>1,\;\al\neq-2$&$f_1=e_4,\,f_2=e_1,\,f_3=-\frac{x}{2}e_1+e_3,\,f_4=-\frac{1}{\vert\al\vert} e_2$& $\mathfrak{r}_{4,-\frac{1}{\vert\al\vert},\frac{1}{\vert\al\vert}}$\\
					\hline
					$\B_{1,\al}^2$&$ \vert\al\vert<1,\; \al\neq0$&$f_1=e_1,\,f_2=e_3,\,f_3=-\frac{x}{\al}e_2+e_4,\,f_4=\frac{-x+\al}{\al}e_2+e_4$& $\mathfrak{r}_{4,-1,-\al}$\\
					\hline
					$\B_{1,\al}^2$&$ \vert\al\vert>1,\; \al\neq-2$&$f_1=-\frac{x}{ \vert\al\vert} e_2+e_4,\,f_2=e_1,\,f_3=e_3,\,f_4=-\frac{x+1}{ \vert\al\vert}e_2+e_4$& $\mathfrak{r}_{4,-\frac{1}{\vert\al\vert},\frac{1}{\vert\al\vert}}$\\
					\hline			
					$\B_{1,-2}^1$&$ x,y\neq0$& $f_1=-ye_1+e_2+\frac{1}{x}e_4 ,\; f_2=-\frac{1}{2}ye_1+e_3 ,\; f_3=ye_1 ,\; f_4=e_2+e_3$&$\mathfrak{d}_{4,2}$\\
					\hline
					$\B_{1,-2}^1$&$ x\neq0,\; y=0$& $f_1=e_2+\frac{1}{x}e_4 ,\; f_2=e_3 ,\; f_3=e_1 ,\; f_4=e_2+\frac{1}{2x}e_4 $&$\mathfrak{r}_{4,-\frac{1}{2},\frac{1}{2}}$\\
					\hline
					$\B_{1,-2}^1$&$ x=0$& $f_1=e_4 ,\; f_2=-\frac{y}{2}e_1+e_3 ,\; f_3=e_1 ,\; f_4=\frac{1}{2}e_2  $&$\mathfrak{r}_{4,-\frac{1}{2},\frac{1}{2}}$\\			
					\hline
					$\B_{1,-2}^2$&$ x\neq0$&$f_1=e_4, \; f_2=\frac{1}{2x}e_1+e_2+\frac{1}{x} e_3, \; f_3=-xe_1, \; f_4=-\frac{1}{x}e_3+(\frac{y+1}{x^2})e_4$&$\mathfrak{d}_{4,2}$\\
					\hline	
					$\B_{1,-2}^2$&$ x=0$&$f_1=e_4, \; f_2=-\frac{y}{2}e_1+e_3, \; f_3=e_1, \; f_4=\frac{1}{2}e_2$&$\mathfrak{r}_{4,-\frac{1}{2},\frac{1}{2}} $\\
					\hline	
					$\B_{1,-1}^1$&&$f_1=-\frac{1}{2}ye_1-xe_2+e_{3} ,\; f_2=e_1 ,\; f_3=e_4 ,\; f_4=-e_2$&$\mathfrak{r}_{4,-1,-1}$\\
					\hline
					$\B_{1,0}^1$&&$f_1=e_2,\,f_2=e_1,\,f_3=-\frac{x}{2}e_1+e_3,\,f_4=e_4$&$\mathfrak{r}\mathfrak{r}_{3,-1}$\\
					\hline	
					$\B_{1,0}^2$&&$f_1=e_2,\,f_2=e_1,\,f_3=e_3,\,f_4=xe_2+e_4$&$\mathfrak{r}\mathfrak{r}_{3,-1}$\\
					\hline	
					$\B_{1,0}^3$&&$f_1=-e_2,\,f_2=-\frac{x}{2} e_1+e_{3},\,f_3=e_1,\,f_4=e_4$&$\mathfrak{r}\mathfrak{r}_{3,-1}$\\
					\hline
					$\B_{1,1}^1$& $x\neq0$& $f_1=e_{1}, \;f_2=xe_{3}-\frac{y}{2}e_{4},\; f_{3}=\frac{y}{2}e_{1}+xe_{2}-e_{4} , \; f_4=e_2 $& $\mathfrak{r}_{4,-1,-1}$\\
					\hline
					$\B_{1,1}^2$&& $f_1=e_1, \;f_2=-\frac{x}{2} e_1+e_3, \;f_3=e_{4}, \; f_4=e_2 $& $\mathfrak{r}_{4,-1,-1}$\\
					\hline
					$\B_{2}$& &$f_1=e_1, \;f_2=-e_4, \;f_3=-\frac{x}{2}e_1+e_{3}, \; f_4=e_2 $& $\mathfrak{r}_{4,-1}$\\
					\hline
					$\B_{3,\alpha}^1$&$ \frac{1}{\alpha}>\frac{1}{2}$& $f_1=e_1, \;f_2=-xe_1+\frac{\alpha-2}{\alpha} e_3, \;f_3=-\frac{\alpha-2}{\alpha}e_4, \; f_4=-e_2 $& $\mathfrak{d}_{4,\frac{1}{\alpha}}$\\
					\hline
					$\B_{3,\alpha}^1$&$ \frac{1}{\alpha}<\frac{1}{2}$& $f_1=-xe_1+\frac{\alpha-2}{\alpha}e_3, \;f_2=e_1, \;f_3=\frac{\alpha-2}{\alpha}e_4, \; f_4=-e_2 $& $\mathfrak{d}_{4,\frac{\alpha-1}{\alpha}}$\\
					\hline
					$\B_{3,2}^1$&$ x\neq0$&$f_1=-xe_1, \;f_2=e_3, \;f_3=xe_4, \; f_4=-e_2 $& $\mathfrak{h}_{4}$\\
					\hline
					$\B_{3,2}^1$&$ x=0$&$f_1=e_3, \;f_2=e_1, \;f_3=e_4, \; f_4=-e_2 $& $\mathfrak{d}_{4,\frac{1}{2}}$\\
					\hline
					$\B_{3,\alpha}^2$&$\frac{1}{\alpha}>\frac{1}{2}$& $f_1=e_1+x \alpha e_2-e_4, \;f_2=e_3, \;f_3=x \alpha e_2-e_4, \; f_4=-e_2+\frac{\alpha-1}{\alpha}e_3 $& $\mathfrak{d}_{4,\frac{1}{\alpha}}$\\
					\hline
					$\B_{3,\alpha}^2$&$\frac{1}{\alpha}<\frac{1}{2}$& $f_1=-x \alpha e_2+e_3+e_4, \;f_2=e_1, \;f_3=-x \alpha e_2+e_4, \; f_4=\frac{1}{\alpha}e_1-e_2 $& $\mathfrak{d}_{4,\frac{\alpha -1}{\alpha}}$\\
					\hline
					$\B_{3,2}^2$&& $f_1=e_3, \;f_2=e_1, \;f_3=-2x e_2+e_4, \; f_4=-e_2 $& $\mathfrak{d}_{4,\frac{1}{2}}$\\
					\hline
					$\B_{3,\frac{1}{2}}^1$& &$f_1=e_1, \;f_2=\frac{1}{3}y e_1-\frac{1}{2}xe_2+e_3, \;f_3=-e_4, \; f_4=-e_2 $& $\mathfrak{d}_{4,2}$\\
					\hline
					$\B_{3,\frac{1}{2}}^2$&$ y\neq0$& $f_1=e_1, \;f_2=-e_3-\frac{x}{y}e_4, \;f_3=xe_1-\frac{1}{2}ye_2+e_4, \; f_4=-\frac{2}{y} e_4 $& $\mathfrak{d}_{4,2}$\\
					\hline
					$\B_{3,1}$&$ x=0,\; y=0$& $f_1=e_1, \;f_2=ze_1+e_3, \;f_3=-e_4, \; f_4=-e_2 $& $\mathfrak{d}_{4,1}$\\
					\hline
					$\B_{3,1}$&$ x\neq0,\; y=0$& $f_1=\frac{z}{x}e_1-e_2+\frac{1}{x}e_3, \;f_2=e_4, \;f_3=-\frac{z}{x} e_1-\frac{1}{x}e_3, \; f_4=e_1-\frac{1}{x}e_4 $& $\mathfrak{r}_{2}\mathfrak{r}_{2}$\\
					\hline
					$\B_{3,1}$&$ x=0,\; y\neq0,\; z\neq0$& $f_1=-\frac{\sqrt{yz}}{2y}e_1-\frac{1}{2} e_2-\frac{1}{2\sqrt{yz}}e_3, \;f_2=-\sqrt{yz}e_1-ye_2+ e_4, \;f_3=\frac{\sqrt{yz}}{2y}e_1-\frac{1}{2} e_2+\frac{1}{2\sqrt{yz}}e_3, \; f_4=\sqrt{yz}e_1-ye_2+ e_4 $& $\mathfrak{r}_{2}\mathfrak{r}_{2}$\\
					\hline
					$\B_{3,1}$&$x=0,\; y\neq0,\; z=0$& $f_1=e_1, \,f_2=e_3, \,f_3=ye_2-e_4, \, f_4=-e_2 $& $\mathfrak{d}_{4,1}$\\
					\hline
					$\B_{3,1}$&$xy\neq0,\, x^{2}+4yz=0$& $f_1=-ye_2+e_4, \,f_2=-\frac{x^{2}}{4y}e_1-\frac{x}{2}e_2+e_3, \,f_3=\frac{x^{2}}{4}e_1+\frac{xy}{2} e_2-\frac{x}{2}e_4, \, f_4=-e_2 $& $\mathfrak{d}_{4,1}$\\
					\hline		
					$\B_{3,1}$&$ xy\neq0,\, x^{2}+4yz>0$& $f_1=\frac{z}{\sqrt{x^{2}+4yz}}e_1-\frac{x+\sqrt{x^{2}+4yz}}{2\sqrt{x^{2}+4yz}}e_2+\frac{1}{\sqrt{x^{2}+4yz}}e_3, \,f_2=(-x+\sqrt{x^{2}+4yz})e_1-2ye_2+2e_4, \,f_3=-\frac{z}{\sqrt{x^{2}+4yz}}e_1+\frac{x-\sqrt{x^{2}+4yz}}{2\sqrt{x^{2}+4yz}}e_2-\frac{1}{\sqrt{x^{2}+4yz}}e_3, \, f_4=-\frac{x+\sqrt{x^{2}+4yz}}{2}e_1-ye_2+e_4 $& $\mathfrak{r}_{2}\mathfrak{r}_{2}$\\
					\hline
					$\B_{3,1}$&$ xy\neq0,\, x^{2}+4yz<0$
					& $f_1=\frac{\sqrt{-x^{2}-4yz}}{2}e_1-e_2, \,f_2=-\frac{x\sqrt{-x^{2}-4yz}-4z}{2\sqrt{-x^{2}-4yz}} e_1-\frac{y\sqrt{-x^{2}-4yz}+x}{\sqrt{-x^{2}-4yz}}e_2+\frac{2}{\sqrt{-x^{2}-4yz}} e_3+e_4, \,f_3=\frac{\sqrt{-x^{2}-4yz}}{2}e_1, \, f_4=-\frac{x}{2}e_1-ye_2+e_4 $& $\mathfrak{r}_{2}'$\\
					\hline
					$\B_{4}$&$ x=0$& $f_1=-e_1, \,f_2=ye_1+e_3-e_4, \,f_3=e_4, \, f_4=- e_2 $& $\mathfrak{d}_{4,1}$\\
					\hline
					$\B_{4}$&$ x\neq0$& $f_1=\frac{x-y}{x} e_1-\frac{1}{x}e_3, \,f_2=-xe_1+e_4, \,f_3=\frac{y}{x} e_1-e_2+\frac{1}{x}e_3, \, f_4=e_4 $& $\mathfrak{r}_{2}\mathfrak{r}_{2}$\\
					\hline
					$\B_{5,1}^+$&& $f_1=-xe_2-e_3+e_4, \,f_2=e_1, \,f_3=e_3, \, f_4= e_1+e_2 $& $\mathfrak{d}_{4,2}$\\
					\hline
					$\B_{5,2}^+$&$ x\neq0$& $f_1=\frac{1}{2} e_2-\frac{2}{x} e_4, \,f_2=e_3, \,f_3=\frac{x}{2}e_1 +e_3, \, f_4=\frac{4}{x} e_4 $& $\mathfrak{d}_{4,2}$\\
					\hline
					$\B_{5,2}^+ $&$ x=0$& $f_1=e_3+e_4, \,f_2=e_1, \,f_3=e_3, \, f_4=- e_1+e_2 $& $\mathfrak{d}_{4,2}$\\
					\hline
					$\B_{5,3}^+ $&$ x\neq0$& $f_1=-e_3-e_4, \,f_2=e_1, \,f_3=xe_1+xe_2+e_3, \, f_4=- \frac{1}{x} e_3 $& $\mathfrak{d}_{4,2}$\\
					\hline
					$\B_{5,3}^+ $&$ x=0$& $f_1=e_4, \,f_2=e_1, \,f_3=e_3, \, f_4= e_2 $& $\mathfrak{d}_{4,2}$\\
					\hline
					$\B_{5,4}^+ $&$ x\neq0$&$f_1=e_3+e_4, \,f_2=e_1, \,f_3=-xe_1+xe_2+e_3, \, f_4=- e_1+e_2 $ & $\mathfrak{d}_{4,2}$\\
					\hline
					$\B_{5,4}^+$&$ x=0$& $f_1=-e_3+e_4, \,f_2=e_1, \,f_3=e_3, \, f_4=e_1+e_2 $& $\mathfrak{d}_{4,2}$\\
					\hline
					$\B_{5,1}^- $&& $f_1=-xe_2+e_4, \,f_2=e_1, \,f_3=-e_3, \, f_4=e_2 $& $\mathfrak{d}_{4,2}$\\
					\hline
					$\B_{5,2}^- $&$ x\neq0$& $f_1=e_2+\frac{4}{x} e_4, \,f_2=e_3, \,f_3=xe_1+2e_3, \, f_4=-\frac{4}{x} e_4 $& $\mathfrak{d}_{4,2}$\\
					\hline
					$\B_{5,2}^- $&$ x=0$& $f_1=e_3+e_4, \,f_2=e_1, \,f_3=-e_3, \, f_4= e_1+e_2 $& $\mathfrak{d}_{4,2}$\\
					\hline
					\caption{Isomorphisms from the Lie algebras obtained in Table \ref{tab4}  onto the  Lie algebras in Table \ref{tab1}}
					\label{tab6}
				\end{longtable}	
		\end{center}}
		
		{\renewcommand*{\arraystretch}{1.4}	
			\begin{center}
				\begin{longtable}{|ll|p{8cm}|c|}
					\hline
					Source && \multicolumn{1}{c|}{Isomorphism}&Target \\
					\hline
					$\mathcal{C}_{1,1} $&$ -1\le \alpha <1$& $f_1=e_1, \;f_2=e_3, \;f_3=e_2, \; f_4=-e_4 $& $\mathfrak{r}_{4,-1,\alpha}$\\
					\hline
					$\mathcal{C}_{1,1} $&$ \alpha <-1$& $f_1=e_2, \;f_2=e_1, \;f_3=e_3, \; f_4=-\frac{1}{\alpha}e_4 $& $\mathfrak{r}_{4,\frac{1}{\alpha},-\frac{1}{\alpha}}$\\
					\hline
					$\mathcal{C}_{1,1} $&$ \alpha >1$& $f_1=e_2, \;f_2=e_3, \;f_3=e_1, \; f_4=-\frac{1}{\alpha}e_4 $& $\mathfrak{r}_{4,-\frac{1}{\alpha},\frac{1}{\alpha}}$\\
					\hline
					$\mathcal{C}_{1,1} $&$ \alpha=1$& $f_1=e_3, \;f_2=e_1, \;f_3=e_2, \; f_4=e_4 $& $\mathfrak{r}_{4,-1,-1}$\\
					\hline
					$\mathcal{C}_{1,2}$&& $f_1=e_3, \;f_2=-e_2, \;f_3=e_1, \; f_4=e_4 $& $\mathfrak{r}_{4,-1}$\\
					\hline
					$\mathcal{C}_{1,3}$&$ 0<\alpha \leq 2$& $f_1=e_2-e_3, \;f_2=e_1, \;f_3=e_2, \; f_4=\frac{\alpha-1}{\alpha}e_1 -e_4 $& $\mathfrak{d}_{4,\frac{1}{\alpha}}$\\
					\hline
					$\mathcal{C}_{1,3}$&$ \alpha <0\; or \; \alpha>2$& $f_1=e_1+\alpha e_2, \;f_2=e_3, \;f_3=e_2, \; f_4=e_3 -e_4 $& $\mathfrak{d}_{4,\frac{\alpha-1}{\alpha}}$\\
					\hline
					$\mathcal{C}_{1,4}$&& $f_1=e_3, \;f_2=e_1-e_2, \;f_3=-e_2, \; f_4=-e_4 $& $\mathfrak{d}_{4,1}$\\
					\hline
					$\mathcal{C}_{1,5}^{+}$&& $f_1=e_2, \;f_2=e_3, \;f_3=e_1, \; f_4=e_4 $& $\mathfrak{d}_{4,2}$\\
					\hline
					$\mathcal{C}_{1,5}^{-}$&& $f_1=e_2, \;f_2=-e_3, \;f_3=e_1, \; f_4=e_4 $& $\mathfrak{d}_{4,2}$\\
					\hline
					$\mathcal{C}_{1,6}$&& $f_1=-e_4, \;f_2=e_2, \;f_3=e_3, \; f_4=e_1 $& $\mathfrak{r} \mathfrak{r}_{3,0}$\\
					\hline
					$\mathcal{C}_{1,7}$&& $f_1=e_1, \;f_2=e_4, \;f_3=e_2, \; f_4=e_3 $& $\mathfrak{r} \mathfrak{h}_{3}$\\
					\hline
					$\mathcal{C}_{1,8}$&& $f_1=e_1, \;f_2=e_3, \;f_3=e_2, \; f_4=-e_4 $& $\mathfrak{d}_{4,1}$\\
					\hline
					$\mathcal{C}_{1,9}$&& $f_1=\frac{1}{2} e_3-\frac{1}{2}e_4, \;f_2=e_1-e_2, \;f_3=-\frac{1}{2} e_3-\frac{1}{2}e_4, \; f_4=e_1+e_2 $& $\mathfrak{r}_{2}\mathfrak{r}_{2}$\\
					\hline
					$\mathcal{C}_{1,10}$&& $f_1=-e_4, \;f_2=e_3, \;f_3=e_1, \; f_4=-e_2 $& $\mathfrak{r}_{2}'$\\
					\hline
					$\mathcal{C}_{2,1}$&$ x=0$& $f_1=-e_2, \;f_2=-ye_2+e_4, \;f_3=e_1, \; f_4=e_3 $& $\mathfrak{r}\mathfrak{r}_{3,0}$\\
					\hline
					$\mathcal{C}_{2,1} $&$ x\neq0$& $f_1=-\frac{1}{x} e_3, \;f_2=e_1, \;f_3=-e_2, \; f_4=-xye_2+xe_4 $& $\mathfrak{r}_{2}\mathfrak{r}_{2}$\\
					\hline
					$\mathcal{C}_{2,2} $&$ x\neq0$& $f_1=-\frac{1}{x} e_3, \;f_2=e_1, \;f_3=\frac{y}{x}e_1 -e_2, \; f_4=e_4 $& $\mathfrak{r}_{2}\mathfrak{r}_{2}$\\
					\hline
					$\mathcal{C}_{2,2} $&$ x=0,\; y\neq0$& $f_1=e_4, \;f_2=-ye_1, \;f_3=e_3, \; f_4=-e_2 $& $\mathfrak{r}_{4,0}$\\
					\hline
					$\mathcal{C}_{2,2}$&$ x=0,\; y=0$& $f_1=-e_2, \;f_2=e_4, \;f_3=e_1, \; f_4=e_3 $& $\mathfrak{r}\mathfrak{r}_{3,0}$\\
					\hline
					$\mathcal{C}_{2,3} $&$ x\neq0$& $f_1=\frac{y}{x}e_1 -e_2, \;f_2=-xe_1+ e_4, \;f_3=e_2-\frac{1}{x}e_3 ,\; f_4=e_1 $& $\mathfrak{r}_{2}\mathfrak{r}_{2}$\\
					\hline	
					$\mathcal{C}_{2,3} $&$ x=0,\; y\neq0$& $f_1=e_4, \;f_2=-ye_1, \;f_3=e_3 ,\; f_4=-e_2 $& $\mathfrak{r}_{4,0}$\\
					\hline
					$\mathcal{C}_{2,3} $&$ x=0,\; y=0$& $f_1=e_1-e_2, \;f_2=e_4, \;f_3=e_1 ,\; f_4=e_3 $& $\mathfrak{r}\mathfrak{r}_{3,0}$\\
					\hline
					$\mathcal{C}_{3,1} $&$z=0 ,\; x\neq0$& $f_1=e_1, \;f_2=-\frac{1}{x}e_4, \;f_3=-\frac{y}{x}e_1+e_2 ,\; f_4=-\frac{1}{x}
					e_3 $& $\mathfrak{r}_{4,0}$\\
					\hline
					$\mathcal{C}_{3,1} $&$z=0 ,\; x=0$& $f_1=-e_2, \;f_2=e_2-e_3, \;f_3=ye_1-e_4 ,\; f_4=
					e_1 $& $ \mathfrak{r} \mathfrak{h}_{3}$\\
					\hline
					$\mathcal{C}_{3,1} $&$z\neq0 ,\; x=0$& $f_1=\frac{1}{z} e_3, \;f_2=e_4, \;f_3=-2ye_1-2ze_2+e_4 ,\; f_4=
					e_1 $& $ \mathfrak{r} \mathfrak{r}_{3,-1}$\\
					\hline
					$\mathcal{C}_{3,1} $&$z\neq0,\; -1\leq \frac{-x}{z}<1$& $f_1=e_4, \;f_2=\frac{2yz}{x-z}e_1-2ze_2+e_4, \;f_3=e_1 ,\; f_4=\frac{1}{z}
					e_3 $& $ \mathfrak{r}_{4,-1,-\frac{x}{z}}$\\
					\hline
					$\mathcal{C}_{3,1} $&$z\neq0,\; x\neq0,\;  \frac{-x}{z}>1$& $f_1=e_1, \;f_2=\frac{2yz}{x-z}e_1-2ze_2+e_4, \;f_3=e_4, \; f_4=-\frac{1}{x}
					e_3 $& $ \mathfrak{r}_{4,\frac{z}{x},-\frac{z}{x}}$\\
					\hline
					$\mathcal{C}_{3,1} $&$z\neq0,\; x\neq0,\;  \frac{-x}{z}<-1$& $f_1=e_1, \;f_2=e_4, \;f_3= \frac{2yz}{x-z}e_1-2ze_2+e_4, \; f_4=-\frac{1}{x}
					e_3 $& $ \mathfrak{r}_{4,-\frac{z}{x},\frac{z}{x}}$\\
					\hline
					$\mathcal{C}_{3,1} $&$z\neq0,\;  \frac{-x}{z}=1$& $f_1=-\frac{y}{2x}e_1+e_{2}+\frac{1}{2x}e_{4}, \;f_2=e_4, \;f_3= e_1, \; f_4=\frac{1}{x}
					e_3 $& $ \mathfrak{r}_{4,-1,-1}$\\
					\hline
					$\mathcal{C}_{3,2}$&$z=0 ,\; x=0$& $f_1=-e_2, \;f_2=e_2-e_3, \;f_3=ye_1-e_4 ,\; f_4=
					e_1 $& $ \mathfrak{r} \mathfrak{h}_{3}$\\
					\hline				
					$\mathcal{C}_{3,2} $&$z=0 ,\; x\neq0$& $f_1=e_4, \;f_2=-\frac{y}{x}e_1+ e_2+\frac{1}{x} e_4, \;f_3=-xe_1 \; f_4=
					-\frac{1}{x} e_3 $& $\mathfrak{d}_{4,1}$\\
					\hline
					$\mathcal{C}_{3,2} $&$z\neq0 ,\; x=0$& $f_1=\frac{1}{z} e_3, \;f_2=e_4, \;f_3=-2ye_1-2ze_2+e_4 \; f_4=
					e_1 $& $ \mathfrak{r} \mathfrak{r}_{3,-1}$\\
					\hline	
					$\mathcal{C}_{3,2} $&$z\neq0 ,\; x\neq0,\; \frac{z}{x}=\frac{1}{2}$& $f_1=-\frac{1}{2z} e_4, \,f_2=-\frac{y}{z} e_1+e_2, \,f_3=e_1 \, f_4=-\frac{1}{2z}
					e_3 $& $\mathfrak{h}_{4}$\\
					\hline
					$\mathcal{C}_{3,2}  $&$\neq0,\; x\neq0,\,  \frac{z}{x}<\frac{1}{2}$& $f_1=\frac{y(x-2z)}{z}e_1+e_4, \,f_2=(x-2z)e_2+e_4, \;f_3= -x(x-2z)e_1,\, f_4=-\frac{y(x-2z)}{x^{2}}
					e_2-\frac{1}{x}e_3 $& $ \mathfrak{d}_{4,\frac{x-z}{x}}$\\
					\hline
					$\mathcal{C}_{3,2} $&$z\neq0,\, x\neq0,\,\frac{z}{x}>\frac{1}{2}$& $f_1=(x-2z)e_2+e_4, \,f_2=\frac{y(x-2z)}{z}e_1+e_4, \,f_3= x(x-2z)e_1,\, f_4=-\frac{y(x-2z)}{x^{2}}
					e_2-\frac{1}{x}e_3 $& $ \mathfrak{d}_{4,\frac{z}{x}}$\\
					\hline
					$\mathcal{C}_{4,1}$&& $f_1=-ye_1-xe_2+e_3, \;f_2=e_1, \;f_3=-xe_1+e_4 \; f_4=
					-xe_1-e_2+e_4 $& $ \mathfrak{d}_{4,1}$\\
					\hline
					$\mathcal{C}_{4,2}$&& $f_1=e_3, \;f_2=e_1, \;f_3=-xe_2+e_4 \; f_4=
					-e_2 $& $ \mathfrak{d}_{4,1}$\\		
					\hline
					$\mathcal{C}_{5,1}^{+}$&& $f_1=\frac{1}{2} e_1-\frac{1}{2}e_2, \;f_2=(-x+y)e_1+(x-y)e_2-e_3+e_4, \;f_3=-(\frac{1}{2}+x+y)e_1-(\frac{1}{2}+x+y)e_2+e_3+e_4, \; f_4=
					-(x+y)e_1-(x+y)e_2+e_3+e_4$& $\mathfrak{r}_{2} \mathfrak{r}_{2}$\\
					\hline
					$\mathcal{C}_{5,2}^{+}$&& $f_1=\frac{1}{2} e_1-\frac{1}{2}e_2, \,f_2=-xe_2-e_3+e_4, \;f_3=-\frac{1}{2}e_1-\frac{1}{2}e_2,\, f_4=
					-xe_2+e_3+e_4 $& $ \mathfrak{r}_{2}\mathfrak{r}_{2}$\\
					\hline
					$\mathcal{C}_{5,1}^{-}$&& $f_1=ye_1-(x+1)e_2+e_3, \;f_2=-(x+1)e_1-ye_2+e_4, \;f_3=ye_1-xe_2+e_3 \; f_4=
					-xe_1-ye_2+e_4 $& $ \mathfrak{r}_{2}'$\\
					\hline
					$\mathcal{C}_{5,2}^{-}$&& $f_1=-e_2-e_3, \;f_2=e_1-xe_2+e_4, \;f_3=-e_3 \; f_4=
					-xe_2+e_4 $& $ \mathfrak{r}_{2}'$\\
					\hline
					\caption{Isomorphisms from the Lie algebras obtained in Table \ref{tab5}   onto the  Lie algebras in Table \ref{tab1}}
					\label{tab7}
				\end{longtable} 
		\end{center}}

		\section*{Acknowledgments:} The authors would like to thank sincerely  Professor Mohamed Boucetta for his many suggestions which were of great help to improve our work.

	\end{document}